\documentclass[]{article}
\usepackage{amsfonts}
\usepackage{ulem}

\usepackage{amssymb}
\usepackage{amsmath}
\usepackage{amsthm}
\usepackage{xcolor}
\usepackage{mathrsfs}
\usepackage[numbers,sort&compress]{natbib}

\usepackage{verbatim}

\allowdisplaybreaks

\allowdisplaybreaks

\def\R{{\mathbb R}}

\def\R{{\mathbb R}}

\def\z+{{\mathbb Z}_+}

\def\supp{{\rm{\,supp\,}}}

\def\bint{{\ifinner\rlap{\bf\kern.30em--}
\int\else\rlap{\bf\kern.35em--}\int\fi}\ignorespaces}

\def\sbint{{\ifinner\rlap{\bf\kern.32em--}
\hspace{0.078cm}\int\else\rlap{\bf\kern.45em--}\int\fi}\ignorespaces}

\newtheorem{theorem}{Theorem}[section]
\newtheorem{lemma}[theorem]{Lemma}

\newtheorem{proposition}[theorem]{Proposition}
\theoremstyle{definition}

\newtheorem{remark}[theorem]{Remark}

\newtheorem{definition}[theorem]{Definition}
\numberwithin{equation}{section}

\numberwithin{equation}{section}

\numberwithin{equation}{section}

\begin{document}

\arraycolsep=1pt

\title{\Large\bf Estimates of fractional Hausdorff operators on weighted Lebesgue and Hardy spaces\footnotetext{\hspace{-0.35cm} {\it 2020
Mathematics Subject Classification}. {42B30, 42B10, 42B20}
\endgraf{\it Key words and phrases.} fractional Hausdorff operator, Hardy space, Lebesgue space, weight
\endgraf This project is supported by Natural Science Foundation of Xinjiang Province (No.2024D01C40) and NSFC (No.12261083).
\endgraf $^\ast$\,Corresponding author.
}}
\author{Zifei Yu and Baode Li$^\ast$}
\date{ }
\maketitle

\vspace{-0.8cm}

\begin{center}
\begin{minipage}{13cm}\small
{\noindent{\bf Abstract.}} In this article, we obtain some necessary and sufficient conditions for the boundedness of fractional Hausdorff operators $h_{\Phi,\beta}$ on weighted Lebesgue spaces $(0\leq\beta<1)$, which are fractional variants of Bandaliev-Safarova [Hacet. J. Math. Stat., 2021, 50, 1334-1346]; it is found that a new constraint for $\beta$ should be added and it holds automatically for non-fractional variants in [HJMS, 2021] $(\beta=0)$.  Then, we further obatin the boundedness of fractional Hausdorff operators $h_{\Phi,\beta}$ on power-weighted Hardy spaces which are fractional variants of Ruan-Fan [Math. Nachr., 2017, 290, 2388-2400]. Ruan-Fan obtained the relevant boundedness results by means of the radial maximal function characterization of Hardy spaces, while in this paper, two different relevant results are obtained respectively by using the radial maximal function characterization and the Riesz characterization of power-weighted Hardy spaces.  

\end{minipage}
\end{center}

\section{Introduction\label{s1}}

The Hausdorff operator originates from the Hausdorff summation method \cite{Hardy1952}. Hausdorff operators are a type of averaging operation, invariant under dilations, capable of achieving constructive approximation, with adjoint operators of the same type, and “almost commuting” with the Fourier transform \cite{Liflyand2020}. Fix a locally integrable function on $(0, \infty)$. The definition of the one-dimensional Hausdorff operator is
$$
h_{\Phi}(f)(x):=\int_0^\infty \frac{\Phi(y^{-1}x)}{y} f(y) dy.
$$
For convenience, we first assume that \( f \) belongs to the Schwartz function space $\mathscr{S}(\R)$. By choosing an appropriate \(\Phi\), the Hausdorff operator includes many classical operators. For example, by choosing \(\Phi\) as \(\chi_{(1,\infty)}(t)t^{-1}\), \(\chi_{(0,1)}(t)\), \( \chi_{(1,\infty)}(t)t^{-1}+\chi_{(0,1]}(t) \) and \(\gamma(1-t)^{\gamma-1}\chi_{(0,1)}(t)\) $(\gamma>0)$ respectively, the Hausdorff operator becomes the Hardy, adjoint Hardy, Ces\'aro and Hardy-Littlewood-P\'olya operators. Most of the results on Hausdorff operators in recent decades can be found in \cite{Liflyand2024}.

A high-dimensional generalization of \( h_{\Phi} \) is
$$
H_\Phi(f)(x):=\int_{\R^n} \frac{\Phi(|y|^{-1}x)}{|y|^n}f(y) dy.
$$
The fractional Hausdorff operator is a natural generalization of the Hausdorff operator and is defined as
\[
H_{\Phi,\beta}(f)(x):=\int_{\R^n} \frac{\Phi(|y|^{-1}x)}{|y|^{n-\beta}} f(y) dy.
\]
It was first proposed by Lin and Sun \cite{LinSun2012}. By choosing \(\Phi\) as \(|x|^{\beta-n}\chi_{(1,\infty)}(|x|)\), \(\chi_{(0,1]}(|x|)\) and \(|x|^{\beta-n}\chi_{(1,\infty)}(|x|)+\chi_{(0,1]}(|x|)\) respectively, the fractional Hausdorff operator $H_{\Phi,\beta}$ becomes the fractional Hardy, adjoint Hardy and Hardy-Littlewood-P\'olya operators. In the one-dimensional case, we denote \( h_{\Phi,\beta} := H_{\Phi,\beta} \).

Bandaliyev and Safarova \cite{2021Safa} obtained the necessary and sufficient conditions for the boundedness of \( h_{\Phi} \) on weighted Lebesgue spaces. Gao, Wu and Guo \cite{2015Gao} obtained the sufficient conditions for the boundedness of \( H_{\Phi,\beta} \) on the power-weighted Lebesgue spaces, where \( n \geq 2 \). Next, X. Liu, M. Wei, P. Song, and D. Yang \cite{2024Wei} obtained the sharp reverse inequality for \( H_{\Phi,\beta} \) on the power-weighted Lebesgue spaces. And the  boundedness of the multi-Huasdorff operator on power-weighted Lebesgue spaces are also studied in \cite{Zhaofayou}. Therefore, these results naturally lead us to study the necessary and sufficient conditions for the boundedness of $h_{\Phi,\beta}$ on weighted Lebesgue spaces.

Chen, Fan, Lin and Ruan \cite{2016Chen} obtained the sufficient conditions for the boundedness of \( H_{\Phi,\beta} \) on Hardy spaces, where \( n \geq 2 \). Ruan and Fan \cite{2017Ruan} obtained the sufficient conditions for the boundedness of \( H_{\Phi} \) on power-weighted Hardy spaces, where \( n \geq 1 \). And multidimensional Hausdorff operators on Hardy sapces are also study in \cite{Zhuxiangrong}. Therefore, these results naturally lead us to study the sufficient conditions for the boundedness of $h_{\Phi,\beta}$ on power-weighted Hardy spaces.

For our results on weighted Lebesgue spaces (Theorems \ref{t1.9} and \ref{t1.10} below), when \( \beta = 0 \) and restricted to \( \mathbb{R}_+ \), our results completely include Theorems 3.1 and 3.3 in \cite{2021Safa}. It is found that the constraint \( q > 1/(1 - \beta) \) and \( p' > 1/(1 - \beta) \) need to be added respectively, and these two conditions naturally hold when \( \beta = 0 \). 

In proving the boundedness of fractional Hausdorff operators on power-weighted Hardy spaces, the following two points should be noted. First, Chen, Fan, Lin and Ruan \cite{2016Chen} relied on the Riesz characterization of Hardy spaces when establishing the boundedness of  fractional Hausdorff operators on Hardy spaces. Ruan and Fan \cite{2017Ruan} proved the boundedness of Hausdorff operators on power-weighted Hardy spaces by using power-weighted Hardy spaces defined via radial maximal functions. In our research, we found that the Riesz characterization of weighted Hardy spaces has also been established by Cao, Chang, D. Yang and S. Yang \cite{Cao2016}. Thus, by employing the two methods from \cite{2016Chen} and \cite{2017Ruan}, respectively, we obtained two sufficient conditions for the boundedness of fractional Hausdorff operators on power-weighted Hardy spaces.  Moreover, the presence of \( \beta \) makes the step of applying the Leibniz formula more complicated. Our results do not include the case of \( \beta = 0 \), the case of $\beta=0$ has been studied in \cite{2017Ruan}.


The structure is as follows. In Section \ref{s2}, we present the basic Definitions and notation to be used. In Section \ref{s3}, we state the main results of this article. In Section \ref{s4}, we provide all the proofs.

Finally, we make some conventions on notation. Let $[x]$ be the greatest integer not exceeding the real number $x$ and $p'$ be the conjugate exponent of $p\in (1,\infty)$, namely, $1/p+1/p'=1$. Throughout the whole article, the constant $C$ depends on the parameters of importance and it may vary from line to line. Let $A, B \geq 0$. Then $A \lesssim B$ and $B \gtrsim A$ mean that there exists a constant $C>0$ such that $A \leq C B$, where $C$ depends only on the parameters of importance.

\section{Definitions and notation \label{s2}} 

\begin{definition}\label{def2.6}\cite[Definition 8.2.1]{Loukas2024}
A {\it weight} is a nonnegative locally integrable function on $\mathbb{R}^n$ that vanishes or takes the value $\infty$ only on a set of measure zero. Given a weight $w$ and a measurable set $E$, we denote the $w d x$-measure of $E$ by
$$
w(E):=\int_E w(x) d x.
$$
\end{definition}

\begin{definition}\label{def2.3}\cite[Definition 2.1]{2017Ruan}
Let \(1<p<\infty\). We say that a weight \(w \in A_p\) if there exists a constant \(C>0\) such that for all balls \(B\subset\R^n\),
\[
\left(\frac{1}{|B|} \int_B w(x) d x\right)\left(\frac{1}{|B|} \int_B w(x)^{-1 /(p-1)} d x\right)^{p-1} \leq C .
\]
We say that a weight \(w \in A_1\) if there is a constant \(C>0\) such that for all balls \(B\),
\[
\frac{1}{|B|} \int_B w(x) d x \leq C \underset{x \in B}{\operatorname{essinf}} w(x)
\]
Finally, we define $A_{\infty}:=\bigcup_{1 \leq p<\infty} A_p$.
\end{definition}

It is known that $A_p \subset A_r$ for all $r>p$. We also know that if $w \in A_p$ with $1<p<\infty$, then $w \in A_q$ for some $1<q<p$. Therefore, we may use $r_w:=\inf \left\{q>1: w \in A_q\right\}$ to denote the {\it critical index} of $w$. Obviously, if $w \in A_q, q>1$, then we have $1 \leq q_w<q$.

An important example of $A_p$ weight is the power function $|x|^\alpha$. It is well known that $|x|^\alpha$ is an $A_1$ weight if and only if $-n<\alpha \leq 0$. Also, 
\begin{equation}\label{eq2}
|x|^\alpha \in \cap_{(n+\alpha)/n<p<\infty} A_p\quad\text{if}\quad 0<\alpha<\infty,
\end{equation}
where $(n+\alpha)/n$ is called the critical index of $|x|^\alpha$.

\begin{definition}\label{def2.2}\cite[Definition 1.1.1]{Loukas2008}
For \(1\leq p<\infty\), given a weight function \(w\) on \(\R\), we define the \(L^p_w(\R)\) norm of a measurable function \(f\) by
\[
\|f\|_{L^p_w(\R)}:=\left(\int_{\R}|f(x)|^p w(x) d x\right)^{\frac{1}{p}}.
\]
For any \(1\leq p< \infty\), we define \(L^p_w(\R)\) to be the space of all measurable functions \(f\) with \(\|f\|_{L^p_w(\R)}<\infty\). 

\end{definition}

\begin{definition}\label{def2.10}\cite[Definition 1.2.5]{Loukas2008}
For \(1\leq p<\infty\), given a weight function \(w\) and a $w dx$-measurable function \(f\) on \(\R\), we define the weak \(L^p_w(\R)\) norm of \(f\) by
\[
\|f\|_{L^{p, \infty}_w(\R)}:=\sup _{\lambda>0} \lambda w(\{x \in \R:|f(x)|>\lambda\})^{\frac{1}{p}}
\]
For \(1\leq p<\infty\), the space \(L^{p, \infty}_w(\R)\) is defined as the set of all \(wdx\)-measurable functions \(f\) such that \(\|f\|_{L^{p, \infty}_w(\R)}<\infty\). 
\end{definition}

\begin{definition}\label{def2.7}
Given \(f\) in Schwartz space \(\mathscr{S} \left(\mathbb{R}\right)\) we define
\[
\widehat{f}(\xi):=\int_{\mathbb{R}} f(x) e^{-2 \pi i x \xi} d x .
\]
We call \(\widehat{f}\) the {\it Fourier transform} of \(f\).
\end{definition}

\begin{definition}\label{def2.8}
Given \(f\) in \(\mathscr{S} \left(\mathbb{R}\right)\), we define
\[
f^{\vee}(x):=\widehat{f}(-x)
\]
for all \(x \in \mathbb{R}\). The operator
\[
f \mapsto f^{\vee}
\]
is called the {\it inverse Fourier transform.}
\end{definition}

\begin{definition}\label{def2.9}\cite[Definition 5.1.1]{Loukas2008}
The {\it Hilbert transform} of \(f\in\mathscr{S}(\R)\) is
\[
H(f)(x):=\frac{1}{\pi} \text { p.v. } \int_{-\infty}^{+\infty} \frac{f(y)}{x-y} d y .
\]
A well-known identity is 
\begin{equation}\label{eq2.6}
\widehat{H(f)}(x)=(-i \operatorname{sgn} x)\widehat{f}(x) .
\end{equation}
This formula can be used to give an alternative definition of the Hilbert transform.
\end{definition}

\begin{definition}\label{def2.1}\cite[Page 4]{2017Ruan}
Let \(0<p\leq 1\), $w$ be a weight function on $\R$ and \(\varphi\) be a function in \(\mathscr{S}\left(\R\right)\) satisfying
\[
\int_{\R} \varphi(x) d x \neq 0 .
\]
For any tempered distributions $f$, the {\it radial maximal operator} $M^+_{\varphi}$ of $f$ is defined by
\[
M_{\varphi}^+f(x):=\sup _{0<s<\infty}\left|\varphi_s * f(x)\right|,\quad x\in\R,
\]
where \(\varphi_s(x):=\frac{1}{s} \varphi\left(\frac{x}{s}\right)\).
The {\it weighted Hardy space} \(H_w^p(\R)\) is the space of all tempered distributions \(f\) satisfying
\[
\|f\|_{H_w^p}:=\left\|M_{\varphi}^+f(x)\right\|_{L_w^p}<\infty .
\]
And the {\it weighted weak Hardy space} $H_w^{p, \infty}\left(\mathbb{R}\right)$ is the space of all tempered distributions $f$ satisfying
\[
\|f\|_{H_w^{p, \infty}}:=\|M_{\varphi}^+f(x)\|_{L^{p,\infty}_w}<\infty.
\]
\end{definition}

\begin{definition}\label{def2.4}\cite[Definition 1.4]{Cao2016}
Let $0<p\leq1$. The {\it weighted Hilbert Hardy space} $H^p_{w, \rm{Hilb}}\left(\mathbb{R}\right)$ is defined to be the completion of the set
$$
\mathbb{H}^p_{w, {\rm Hilb}} (\R):=\left\{f \in L^2\left(\mathbb{R}\right):\|f\|_{H^p_{w, \rm{Hilb}} (\R)}<\infty\right\}
$$
under the quasi-norm $\|\cdot\|_{H^p_{w, {\rm Hilb}}\left(\mathbb{R}\right)}$, where, for all $f \in L^2\left(\mathbb{R}\right)$,
$$
\|f\|_{H^p_{w, {\rm Hilb}}\left(\mathbb{R}\right)}:=\|f\|_{L^p_w\left(\mathbb{R}\right)}+\|H(f)\|_{L^p_w\left(\mathbb{R}\right)}.
$$
\end{definition}


\begin{proposition}\label{pro2.1}\cite[Theorem 1.5]{Cao2016}
	Let $0<p\leq 1$ and $w\in A_\infty$, then $H^p_w(\R)=H^p_{w,{\rm Hilb}}(\R)$ with equivalent quasi-norms.
\end{proposition} 

By Proposition \ref{pro2.1}, we obtain the follow Proposition \ref{pro2.5} which implies that space $H^p_w(\R,\varphi)$ is independent of the choice of the Schwartz functions $\varphi$ satisfying $\int_{\R}\varphi(x)dx\neq0$.

\begin{proposition}\label{pro2.5}
Let $w\in A_\infty$. For any $\varphi_1$, $\varphi_2\in \mathscr{S}(\R)$ with $\int_{\R}\varphi_i(x)dx\neq0$, $i=1,2$, $H^p_w(\R,\varphi_1)=H^p_w(\R,\varphi_2)$ with equivalent quasi-norms.
\end{proposition}

\begin{definition}\label{def2.5}\cite[Page 5]{2017Ruan}
Let $0<p\leq 1 \leq r<\infty, w \in A_r, \gamma=\left[n\left(r_w / p-1\right)\right]$ and $1 / r+1 / r^{\prime}=1$, where $r_w$ be the critical index of $w$.
\begin{enumerate}
\item[(i)] {\it The weighted Campanato-Morrey space} $\mathcal{L}_w^{p, r}\left(\mathbb{R}\right)$ is defined to be the collection of all functions satisfying \(\|f\|_{\mathcal{L}_w^{p,r}(\R)}<\infty\), where
$$
\|f\|_{\mathcal{L}_w^{p,r}(\R)}:= \sup _{0<R<\infty, x_0 \in \mathbb{R}^n} \frac{1}{w\left(B\left(x_0, R\right)\right)^{\frac{1}{p}-1}}\left(\int_{B\left(x_0, R\right)}\left|\frac{f(x)-P_{B, f}(x)}{w(x)}\right|^{r^{\prime}} \frac{w(x) d x}{w\left(B\left(x_0, R\right)\right)}\right)^{\frac{1}{r}^{\prime}},
$$
where $P_{B, f}$ is the unique polynomial of degree $\leq m$ such that
$$
\int_{B\left(x_0, R\right)}\left(f(x)-P_{B, f}(x)\right) x^{m} d x=0
$$
for all  $m=0,1,2\cdots$.

\item[(ii)] If $f \in C^m\left(\mathbb{R}\right)$, replacing the polynomial $P_{B, f}(x)$ by the Taylor polynomial
$$
\widetilde{P}_{B, f}(x):=\sum_{k=0}^m \frac{d^m f}{dx^m} \left(x_0\right)\left(x-x_0\right)^{m} /(m!)
$$
in the above (i), we denote the new space by $\widetilde{\mathcal{L}}_w^{p, r}\left(\mathbb{R}\right)$.
\end{enumerate}
\end{definition}

\begin{proposition}\label{rem2.1}\cite[Remark 2.3]{2017Ruan}
If $0<p\leq 1 \leq r<\infty$ and $w \in A_r$, then we have the following facts:
\begin{enumerate}
\item[(i)] The dual space of $H_w^p\left(\mathbb{R}\right)$ is the space $\mathcal{L}_w^{p, r}\left(\mathbb{R}\right)$ and they have an $\left(H_w^p, \mathcal{L}_w^{p, r}\right)$ pair inequality
$$
|\langle f, g \rangle | \lesssim \|f\|_{H_w^p(\R)}\|g\|_{\mathcal{L}_w^{p, r}(\R)}
$$
for any $f \in H_w^p\left(\mathbb{R}\right)$ and $g \in \mathcal{L}_w^{p, r}\left(\mathbb{R}\right)$.

\item[(ii)] $\widetilde{\mathcal{L}}_w^{p, r}\left(\mathbb{R}\right) \subset \mathcal{L}_w^{p, r}\left(\mathbb{R}\right)$ and
$$
\|f\|_{\mathcal{L}_w^{p, r}(\R)} \lesssim \|f\|_{\widetilde{\mathcal{L}}_w^{p, r}(\R)}.
$$
\end{enumerate}
\end{proposition}

\section{Main results \label{s3}}

In this section,  we will discuss the boundedness of $h_{\Phi,\beta}$ on weighted Lebesgue spaces and power-weighted Hardy spaces.

\subsection{$h_{\Phi,\beta}:L^p_v(\R)\to L^q_u(\R)$ \label{s3.1}}

\begin{theorem}\label{t1.9}
Let \(1 < p,q < \infty\), \(0\leq\beta<1\), \(q > \frac{1}{1 - \beta}\) and \(\frac{1}{p}-\frac{1}{q}=\beta\). Let \(u\) and \(v\) be weight functions on \(\mathbb{R}\), which are even functions and increasing on \(\mathbb{R}_+\). Let \(\Phi\) be a non-negative function on \(\mathbb{R}\) satisfying
\begin{equation}\label{eq3.18}
 \int_{|t|\leq1} \Phi(t)^{\frac{1}{1-\beta}} |t|^{\frac{1}{q(1-\beta)}-1} dt 
< \infty,
\end{equation}
and for any \(|t|\geq1\), there exist constants \(C_1\geq0\) and \(C_2\geq0\) such that
\begin{equation}\label{eq2.3}
\frac{C_1}{|t|^{1-\beta}}\leq \Phi(t) \leq \frac{C_2}{|t|^{1-\beta}}.
\end{equation}
Then for all \(f\in L^p_v(\mathbb{R})\), the inequality
\begin{equation}\label{eq2.2}
\|h_{\Phi,\beta}(f)\|_{L^q_u(\mathbb{R})}\leq C \|f\|_{L^p_v(\mathbb{R})}
\end{equation}
holds if and only if
\begin{equation}\label{eq3.9}
    A:=
    \sup_{\alpha>0} \left(\int_{|x|\geq\alpha} \frac{u(x)}{|x|^{q(1-\beta)}}dx\right)^\frac{1}{q} \left( \int_{|x|\leq\alpha} v(x)^{1-p'} dx\right)^\frac{1}{p'}<\infty.
\end{equation}
Besides, if \(C>0\) is the best constant in \((\ref{eq2.2})\), then
\[
C_1 A \leq C \leq 2^\frac{1}{q'} A\left\{  K_{\Phi,\beta,q}2^{\beta-1} [q(1-\beta)-1]^\frac{1}{q}(1+2^\frac{1}{q'})+C_2  (p')^\frac{1}{p'}p^\frac{1}{q} \right\} ,
\]
where 
$$
K_{\Phi,\beta,q}:=\left( \int_{\mathbb{R}} \Phi(t)^{\frac{1}{1-\beta}} |t|^{\frac{1}{q(1-\beta)}-1} dt \right)^{1-\beta}<\infty.
$$
\end{theorem}

\begin{theorem}\label{t1.10}
    Let \(1<p, q<\infty\), \(0\leq\beta<1\), \(p'>\frac{1}{1-\beta}\) and \(\frac{1}{p}-\frac{1}{q}=\beta\). Let \(u\) and \(v\) be weight functions on \(\R\), which are even functions and decreasing on \(\R_+\). Let \(\Phi\) be a non-negative function on \(\R\) satisfying
    \begin{equation}\label{eq2.7}
    \int_{|t|\geq1} \Phi(t)^\frac{1}{1-\beta} |t|^{\frac{1}{q(1-\beta)}-1} dt
     <\infty,
	\end{equation}
    and for all \(|t|\leq 1\), there exist constants \(C_1'\geq0\) and \(C_2'\geq0\) such that
    \begin{equation}\label{eq2.4}
        C_1'\leq \Phi(t) \leq C_2'.
    \end{equation}
    Then for all \(f\in L^p_v(\R)\), the inequality
    \begin{equation}\label{eq2.5}
        \|h_{\Phi,\beta}(f)\|_{L^q_u(\R)}\leq C \|f\|_{L^p_v(\R)}
    \end{equation}
    holds if and only if
    \begin{equation}\label{eq3.11}
    B:=
    \sup_{\alpha>0} \left(\int_{|x|\leq\alpha} u(x) dx\right)^\frac{1}{q} \left( \int_{|x|\geq\alpha} \frac{v(x)^{1-p'}}{|x|^{(1-\beta)p'}} dx\right)^\frac{1}{p'}<\infty.
    \end{equation}
    Besides, if \(C>0\) is the best constant in \((\ref{eq2.5})\), then
\[
C_1' B \leq C \leq  2^\frac{1}{q'} B \left\{  K_{\Phi,\beta,q } 2^{\beta-1} [p'(1-\beta)-1]^\frac{1}{p'}  (1+2^\frac{1}{q'}) +C_2'  (p')^\frac{1}{p'} p^\frac{1}{q} \right\}.
\]
\end{theorem}

\begin{remark}\label{rem2.6}
When \(\beta = 0\) and restrict \(\mathbb{R}\) to \(\mathbb{R}_+\), Theorem \ref{t1.9} and Theorem \ref{t1.10} coincide with Theorem 2.3 and Theorem 2.4 in \cite{2021Safa}, respectively.   
\end{remark}

\begin{remark}\label{rem2.9}
$|t|^{\beta-1}\chi_{(1,\infty)}(|t|)$ and $\chi_{(0,1]}(|t|)$ are specific examples of $\Phi$ in Theorem \ref{t1.9} and Theorem \ref{t1.10}.
\end{remark}

\subsection{$h_{\Phi,\beta}:H^p_{|\cdot|^\alpha}(\R)\to H^q_{|\cdot|^\gamma}(\R)$\label{s3.2}}
    Let \(0<p,q<\infty\), \(-1<\alpha,\gamma<\infty\) and \(0<\beta<1\). When $h_{\Phi,\beta}$ is bounded from $H^p_{|\cdot|^\alpha}(\R)$ to $H^q_{|\cdot|^\gamma}(\R)$, the necessary condition of $p,q,\alpha,\gamma$ and $\beta$ are discussed in the following Theorem \ref{t1.2}.

\begin{theorem}\label{t1.2}
    Let \(0<p,q<\infty\), \(-1<\alpha,\gamma<\infty\) and \(0\leq\beta<1\). For any \(f\in H^p_{|\cdot|^\alpha}(\R)\), if
    \[
    \|h_{\Phi,\beta}f\|_{H^q_{|\cdot|^\gamma}(\R)}\lesssim\|f\|_{H^p_{|\cdot|^\alpha}(\R)},
    \]
    then
    \[
    \frac{1+\alpha}{p}-\frac{1+\gamma}{q}=\beta.
    \]
\end{theorem}

\begin{remark}\label{rem2.5}
When \(\beta = 0\), Theorem \ref{t1.2} coincides with one-dimensional case of \cite[Theorem 1.1]{2017Ruan}.
\end{remark}

\subsubsection{Method of radial maximal function\label{s3.2.1}}

Let $p,q,\alpha ,\gamma$ and $\beta$ be as in Theorem \ref{t1.6} below. For the power-weighted Hardy spaces defined via radial maximal operators, we first obtain the boundedness of $h_{\Phi,\beta}$ from $H^p_{|\cdot|^\alpha}(\R)$ to $H^{q,\infty}_{|\cdot|^\gamma}(\R)$, then by using an interpolation, we can further obtain the boundeness of $h_{\Phi,\beta}$ from $H^p_{|\cdot|^\alpha}(\R)$ to $H^q_{|\cdot|^\gamma}(\R)$.


\begin{theorem}\label{t1.6}
Let \(0<p\leq1\), \(0<q<\infty\), \(0\leq\alpha<\infty\), \(-1<\gamma<\infty\), \(0<\beta<1\) and \(\frac{1+\alpha}{p}-\frac{1+\gamma}{q}=\beta\) and let $m=\frac{1+\alpha}{p}-1$ be an integer. Suppose that $\Phi\in L^1(\R)$ with $\widehat{\Phi} \in C^{2m + 1}(\mathbb{R})$ and satisfying
    \begin{equation}\label{eq3.14}
    \int_{\R} \left| \widehat{\Phi}^{(n)}(\xi)\xi^n\right| d\xi<\infty~\text{and}~ \int_{\R} \left| \widehat{\Phi}^{(n+l)}(\xi) \xi^{n+l-m-1}\right|   d\xi<\infty
    \end{equation}
    for all \(n=0,1,\cdots,m\), $k=0,1,\cdots,n$ and \(l=0,1,\cdots,m+1-k\).
    Then, for any \(f\in H^p_{|\cdot|^\alpha}(\mathbb{R})\), we have
    \[
    \|h_{\Phi,\beta}(f)\|_{H^{q,\infty}_{|\cdot|^\gamma}(\R)}\lesssim\|f\|_{H^p_{|\cdot|^\alpha}(\R)}.
    \]
\end{theorem}


\begin{theorem}\label{t1.4}
    Let \(p,~q,~\alpha,~\gamma\) and \(\beta\) be as in Theorem \ref{t1.6}, $p\leq q$ and \(\frac{1+\alpha}{p}-1\) be not restricted to an integer. Let \(m=[\frac{1+\alpha}{p}-1]+1\).  Assume that there exists \(p_0>\max\{\frac{1+\alpha}{1+\beta+\gamma},1+\alpha\}\) such that
    \begin{equation}\label{eq2.1}
         \int_{\R}|\Phi(t)|^s|t|^{(1+\gamma)s/{q_0}-1} dt<\infty,
    \end{equation}
    where \(\frac{1+\alpha}{p_0}-\frac{1+\gamma}{q_0}=\beta\) and \(\frac{1}{s}=1+\frac{1}{q_0}-\frac{1}{p_0}\). Suppose that $\Phi\in L^1(\R)$ with \(\widehat{\Phi}\in C^{2m+1}(\R)\) and satisfying $(\ref{eq3.14})$. Then for all \(f\in H^p_{|\cdot|^\alpha}(\R)\),
    \[
    \|h_{\Phi,\beta}f\|_{H^q_{|\cdot|^\gamma}(\R)}\lesssim \|f\|_{H^p_{|\cdot|^\alpha}(\R)}.
    \]
	\end{theorem}

	\begin{theorem}\label{t1.7}
	Let $p~, q~, \alpha~, \gamma~, m$ and $\beta$ be as Theorem \ref{t1.6}. Suppose that $\Phi\in L^1(\R)$,  $g\left(\xi^2\right):=\widehat{\Phi}(\xi)$, where the function $g\in C^{2m+1}:[0, \infty) \rightarrow \mathbb{R}$ satisfying
	\begin{equation}\label{eq3.42}
	\int_0^{\infty}\left|g^{(j)}(\xi)\right| \xi^{j-\frac{1}{2}} d \xi<\infty~\text{and}~ \int_0^{\infty}\left|g^{(i+j)}(\xi)\right| \xi^{i+j-\frac{m}{2}-1} d \xi<\infty
	\end{equation}
	for all $i, j$, where $i=0,1,\cdots,[\frac{m}{2}+1]$ and $j=0,1,\cdots,[\frac{m+1}{2}]$. Then $h_{\Phi,\beta}$ is bounded from $H_{|\cdot| ^\alpha}^p(\mathbb{R})$ to $H_{|\cdot|^ \gamma}^{q, \infty}(\mathbb{R})$ and
	$$
	\left\|h_{\Phi,\beta}(f)\right\|_{H_{|\cdot|^ \gamma}^{q, \infty}(\R)} \lesssim\|f\|_{H_{|\cdot|^ \alpha}^p(\R)} .
	$$
	\end{theorem}

\begin{remark}\label{rem2.8}
The case of $0<\beta<1$ for $h_{\Phi,\beta}$ is considered in Theorems \ref{t1.6}, \ref{t1.4} and \ref{t1.7}, the case of $\beta=0$ for $h_{\Phi,\beta}$ is considered in \cite[Theorems 2, 3 and 4]{2017Ruan}.
\end{remark}

\subsubsection{Method of Hilbert transform \label{s3.2.2}}
 Let $p,q,\alpha ,\gamma$ and $\beta$ be as in Theorem \ref{t1.6}. Denote by $H$ the Hilbert transform. We first obtain the boundedness of $H\circ h_{\Phi,\beta}$ from $H^p_{|\cdot|^\alpha}$ to $L^{q,\infty}_{|\cdot|^\gamma}$, then by using an interpolation and the characterization of Hilbert transform of $H^q_{|\cdot|^\gamma}(\R)$, we further obtain the boundedness of $h_{\Phi,\beta}$ from $H^p_{|\cdot|^\alpha}(\R)$ to $H^q_{|\cdot|^\gamma}(\R)$. 


\begin{theorem}\label{t1.1}
    Let \(p,~q,~\alpha,~\gamma,~\beta,\) and \(m\) be as in Theorem \ref{t1.6}. Suppose that $\Phi\in L^1(\R)$ with $\widehat{\Phi} \in C^{2m + 1}(\mathbb{R})$ and the support of \(\widehat{\Phi}\) is compact, for any \(f\in H^p_{|\cdot|^\alpha}(\mathbb{R})\), we have
    \[
    \|Hh_{\Phi,\beta}(f)\|_{L^{q,\infty}_{|\cdot|^\gamma}(\R)}\lesssim\|f\|_{H^p_{|\cdot|^\alpha}(\R)} .
    \]
\end{theorem}

\begin{theorem}\label{t1.8}
    Let \(p,~q,~p_0,~q_0,~\alpha,~\gamma,~\beta,~m\), and \(s\) be as Theorem \ref{t1.4}. Suppose that $\Phi\in L^1(\R)$ satisfying \((\ref{eq2.1})\), \(\widehat{\Phi}\in C^{2m+1}(\R)\) and the support of \(\widehat{\Phi}\) is compact. Then for all \(f\in H^p_{|\cdot|^\alpha}(\R)\),
    \[
    \|h_{\Phi,\beta}f\|_{H^q_{|\cdot|^\gamma}(\R)}\lesssim \|f\|_{H^p_{|\cdot|^\alpha}(\R)}.
    \]
\end{theorem}

\begin{remark}\label{rem2.7}
The case of $0<\beta<1$, $\alpha\geq0$ and $\gamma>-1$ for the boundedness of $h_{\Phi,\beta}$ is considered in Theorems \ref{t1.1} and \ref{t1.8}, the case of $\alpha=\gamma=\beta=0$ for $h_{\Phi,\beta}$ is considered in \cite[Theorems 2.2 and 2.3]{2016Chen}.

\end{remark}

\section{Proofs of theorems\label{s4}}

We need some lemmas for preparation. First, let us recall Young's inequality on locally compact groups with Haar measure. Then by Lemma \ref{t3.8}, we further obtain the boundedness of $h_{\Phi,\beta}$ on power-weighted Lebesgue spaces.

\begin{lemma}\cite[Theorem 1.2.12]{Loukas2008}\label{t3.8}
Let $1 \leqslant p, q, s \leqslant \infty$ satisfy $\frac{1}{q}=\frac{1}{p}+\frac{1}{s}-1$, and $\mu$ be a Haar measure on a locally compact group $G$, then
$$
\|f * g\|_{L^q(G, \mu)} \leqslant\|g\|_{L^s(G, \mu)}\|f\|_{L^p(G, \mu)}
$$
for all $f$ in $L^p(G, \mu)$ and for all $g$ in $L^s(G, \mu)$ satisfying $\|g\|_{L^s(G, \mu)}=\|\widetilde{g}\|_{L^s(G, \mu)}$, where $\widetilde{g}(x)=g\left(x^{-1}\right)$ and
$$
(f * g)(x)=\int_G f(y) g\left(y^{-1} x\right) d \mu.
$$
\end{lemma}

\begin{lemma}\label{t3.4}
Let $1 \leq p, q \leq \infty, 0 \leq \beta<1$ and $\alpha, \gamma \in \mathbb{R}$ satisfy $\frac{\gamma+1}{q}=\frac{\alpha+1}{p}-\beta$. If
\[
K_{\Phi, s, q, \gamma}:=\left( \int_{\R}|\Phi(t)|^s|t|^{(1+\gamma)s/q-1} dt \right)^\frac{1}{s}<\infty,
\]
where s satisfies $\frac{1}{q}=\frac{1}{p}+\frac{1}{s}-1$, then the operator
$h_{\Phi, \beta}$ is bounded from $L^p_{|\cdot|^\alpha}\left(\mathbb{R}\right)$ to $L^q_{|\cdot|^\gamma}\left(\mathbb{R}\right)$  , i.e., for all $f \in L^p_{|\cdot|^\alpha}(\R)$
\[
\left\|h_{\Phi, \beta} f\right\|_{L^q_{|\cdot|^\gamma}(\R)} \leq  2^\frac{1}{q'} K_{\Phi, s, q, \gamma}\|f\|_{L^p_{|\cdot|^\alpha}(\R)}.
\] 
In addition, when \(\alpha = \gamma=0\), that is, if
\begin{equation}\label{eq3.8}
K_{\Phi,\beta,q}=\left( \int_{\mathbb{R}} |\Phi(t)|^{\frac{1}{1-\beta}} |t|^{\frac{1}{q(1-\beta)}-1} dt \right)^{1-\beta}<\infty,
\end{equation}
we can obtain that \(h_{\Phi,\beta}\) is bounded from \(L^p(\mathbb{R})\) to \(L^q(\mathbb{R})\), i.e., for all $f\in L^p(\R)$,
$$
\|h_{\Phi,\beta}(f)\|_{L^q(\R)}\leq 2^\frac{1}{q'} K_{\Phi,\beta,q} \|f\|_{L^p(\R)}.
$$  
\end{lemma}

\begin{proof}The proof is based on the idea used in \cite[Proof of Theorem 3.1]{2015Gao}. Split the integration regions into $\mathbb{R}_+$ and $\mathbb{R}_-$. Substituting $-y = \widetilde{y}$, $-x = \widetilde{x}$ for the integrals over $\mathbb{R}_-$, we obtain
$$
\begin{aligned}
\|h_{\Phi,\beta}(f)\|_{L^q_{|\cdot|^\gamma}(\R)}
&=
\left(\int_{-\infty}^\infty \left|\int_{-\infty}^\infty \frac{\Phi(|y|^{-1}x)}{|y|^{1-\beta}} f(y) dy\right|^q |x|^\gamma\right)^\frac{1}{q}\\
&=
\left(\sum_{j=1,-1}\int_0^\infty \left| \sum_{i=1,-1} \int_0^\infty \frac{\Phi(jy^{-1}x)}{y^{1-\beta}} f(iy) dy \right|^q x^\gamma dx\right)^\frac{1}{q}\\
\end{aligned}
$$
Using  $(a+b)^{\frac{1}{q}} \leq a^{\frac{1}{q}} + b^{\frac{1}{q}}$ for $a,b>0$  and Minkowski inequality for $L^q_{|\cdot|^\gamma}(\R)$ norms, we obtain
$$
\|h_{\Phi,\beta}(f)\|_{L^q_{|\cdot|^\gamma}(\R)}
\leq
\sum_{j=1,-1}\sum_{i=1,-1}\left(\int_0^\infty \left| \int_0^\infty \frac{\Phi(jy^{-1}x)}{y^{1-\beta}} f(iy) dy \right|^q x^\gamma dx\right)^\frac{1}{q}.
$$
It is well known that the multiplicative group $\R_+$ is a locally compact group with Haar measure $\frac{dx}{x}$. Thus, by $\frac{1+\alpha}{p}-\frac{1+\gamma}{q}=\beta$ and Lemma \ref{t3.8} with $g(x)=\Phi(jx)(x)^\frac{1+\gamma}{q}$ satisfying $\|g\|_{L^s(\R_+, \frac{dx}{x})}=\|\tilde{g}\|_{L^s(\R_+,\frac{dx}{x})}$, we obtain that
$$
\begin{aligned}
\|h_{\Phi,\beta}(f)\|_{L^q_{|\cdot|^\gamma}(\R)}
&=\sum_{j=1,-1}\sum_{i=1,-1} \left( \int_0^\infty \left| \int_0^\infty \Phi(jy^{-1}x) (y^{-1}x)^\frac{1+\gamma}{q} f(iy) y^{\frac{1+\alpha}{p}} \frac{dy}{y}  \right|^q \frac{dx}{x} \right)^\frac{1}{q}\\
&\leq \sum_{j=1,-1}\sum_{i=1,-1} \left( \int_{0}^\infty |\Phi(jt)|^s t^{\frac{1+\gamma}{q}s-1} dt \right)^\frac{1}{s} \left(\int_0^\infty |f(it)|^p t^\alpha dt \right)^\frac{1}{p}.
\end{aligned}
$$
Substituting \(-t = \widetilde{t}\) for the above integrals involving $j=-1$ and $i=-1$ ,  then by inequality \(a^{\frac{1}{p}} + b^{\frac{1}{p}} \leq 2^{\frac{1}{p'}} (a + b)^{\frac{1}{p}}\), $a^{\frac{1}{s}} + b^{\frac{1}{s}} \leq 2^{\frac{1}{s'}} (a + b)^{\frac{1}{s}}$ for \(a, b > 0\) and $\frac{1}{q}=\frac{1}{p}+\frac{1}{s}-1$, we further obtain 
$$
\begin{aligned}
\|h_{\Phi,\beta}(f)\|_{L^q_{|\cdot|^\gamma}(\R)}
&\leq 
2^\frac{1}{q'} \left(\int_{-\infty}^\infty |\Phi(t)|^s|t|^{\frac{1+\gamma}{q}s-1} dt\right)^\frac{1}{s}\left(\int_{-\infty}^\infty |f(t)|^p |t|^\alpha dt\right)^\frac{1}{p}\\
&=
2^\frac{1}{q'} K_{\Phi,s,q,\gamma} \|f\|_{L^p_{|\cdot|^\alpha}(\R)}.
\end{aligned}
$$
\end{proof}

\begin{remark}\label{t1000}
Lemma \ref{t3.4} is not included in \cite[Theorem 3.1]{2015Gao}, because the authors \cite{2015Gao} only proved the case of \( H_{\Phi,\beta} \) for \( n \geq 2 \).
\end{remark}


Two-weight Hardy inequalities are shown in the following Lemmas \ref{t3.5} and \ref{t3.6}.

\begin{lemma}\label{t3.5}\cite[Theorem 2.1]{1997Kufner}
    Let \(u\) and \(v\) be weight functions on \(\R\). For \(1<p\leq q<\infty\) and \(0\leq\beta<1\), the inequality
    \begin{equation}\label{eq3.6}
    \left[ \int_{\R} \left(\frac{1}{|x|^{1-\beta}}\int_{|y|\leq |x|} f(y) dy \right)^q u(x) dx \right]^\frac{1}{q}\leq C \left(\int_{\R} f(x)^p v(x) \right)^\frac{1}{p}
    \end{equation}
    holds for \(f\geq0\) if and only if \((\ref{eq3.9})\) holds.
    Moreover, if \(C\) is the smallest constant for which \((\ref{eq3.6})\) holds, then
    \[
    A\leq C \leq A(p')^{\frac{1}{p'}}p^\frac{1}{q}.
    \]
\end{lemma}

\begin{lemma}\label{t3.6}\cite[Page 10]{1997Kufner}
    Let \(u\) and \(v\) be weight functions on \(\R\). For \(1<p\leq q<\infty\) and \(0\leq\beta<1\), the inequality
    \begin{equation}\label{eq3.7}
        \left[\int_{\R} \left(\int_{|y|\geq |x|} \frac{f(y)}{|y|^{1-\beta}} dy \right)^q u(x) dx \right]^\frac{1}{q}\leq C \left(\int_{\R} f(x)^p v(x) \right)^\frac{1}{p}
    \end{equation}
    holds for \(f\geq0\) if and only if \((\ref{eq3.11})\) holds.
    Moreover, if \(C\) is the smallest constant for which \((\ref{eq3.7})\) holds, then
    \[
    B\leq C \leq B(p')^{\frac{1}{p'}}p^\frac{1}{q}.
    \]
\end{lemma}

\begin{proof}[Proof of Theorem \ref{t1.9}]
    Necessity. Let \(\alpha>0\) be a fixed number. Suppose that \((\ref{eq2.2})\) holds. Let \(f\) be a function defined on \(\R\) and \({\rm supp}~(f)\subset [-\alpha,\alpha]\). If \(|x|\geq \alpha\) and \(|y|\leq \alpha\), then \(|y|^{-1}|x|=||y|^{-1}x|\geq1\), thus, by \((\ref{eq2.3})\), we have
    
    \begin{align}
        \|h_{\Phi,\beta}(f)\|_{L^q_u(\R)}&=
        \left( \int_{\R}\left|\int_{\R} \frac{\Phi(|y|^{-1}x)}{|y|^{1-\beta}} f(y) dy\right|^q u(x) dx \right)^\frac{1}{q}\nonumber\\ 
        &\geq \left( \int_{|x|\geq \alpha} \left| \int_{|y|\leq \alpha} \frac{\Phi(|y|^{-1}x)}{|y|^{1-\beta}} f(y) dy \right|^q u(x) dx \right)^\frac{1}{q}\nonumber\\
        &\geq C_1 \left( \int_{|x|\geq \alpha} \frac{u(x)}{|x|^{q(1-\beta)}} dx\right)^\frac{1}{q} \left(\int_{|y|\leq \alpha} f(y) dy \right).\label{eq3.19}
    \end{align}

    We choose the test function \(f=f_\alpha (x):=v(x)^{1-p'} \chi_{\{|x|\leq \alpha\}}(x)\), thus
    \begin{equation}\label{eq3.20}
    \|f_\alpha \|_{L^p_v(\R)}=\left(\int_{|x|\leq \alpha} v(x)^{1-p'} dx\right)^\frac{1}{p}.
    \end{equation}
	Then, by $(\ref{eq3.19})$ with $f=f_\alpha$, $(\ref{eq2.2})$ and $(\ref{eq3.20})$, we obtain
    \[
	\begin{aligned}
    C_1 \left(\int_{|x|\geq \alpha} \frac{u(x)}{|x|^{q(1-\beta)}} dx\right)^\frac{1}{q} \left( \int_{|x|\leq \alpha} 
    v(x)^{1-p'} dx \right) 
	&\leq \|h_{\Phi,\beta}(f)\|_{L^q_u(\R)}\\
	&\leq C \|f\|_{L^p_v(\R)}\\
    &=C\left( \int_{|x|\leq \alpha} v(x)^{1-p'} dx \right)^\frac{1}{p},
	\end{aligned}
    \]
    which implies that
    \[
    C_1 \left(\int_{|x|\geq \alpha} \frac{u(x)}{|x|^{q(1-\beta)}} dx\right)^\frac{1}{q} \left( \int_{|x|\leq \alpha} 
    v(x)^{1-p'} dx \right)^\frac{1}{p'}\leq C .
    \]

    Sufficiency. By \((\ref{eq3.18})\), \((\ref{eq2.3})\) and $q(1-\beta)>1$, we have
    \begin{align}
    K_{\Phi,\beta,q}
    &\leq \left( \int_{|t|\leq 1} \Phi(t)^\frac{1}{1-\beta}|t|^{-1+\frac{1}{q(1-\beta)}} dt \right)^{1-\beta} + \left( \int_{|t|\geq 1} \Phi(t)^\frac{1}{1-\beta} |t|^{-1+\frac{1}{q(1-\beta)}} dt \right)^{1-\beta}\nonumber\\
    &\leq \left( \int_{|t|\leq 1} \Phi(t)^\frac{1}{1-\beta} |t|^{-1+\frac{1}{q(1-\beta)}} dt \right)^{1-\beta}+C_2 \left( \int_{|t|\geq1} |t|^{-2+\frac{1}{q(1-\beta)}} dt \right)^{1-\beta}<\infty.\label{eq3.25}
    \end{align}
    Therefore, the condition \((\ref{eq3.8})\) of Lemma \ref{t3.4} is satisfied.

    Same to \cite[Proof of Theorem 3.1]{2021Safa}, without loss of generality, we may assume that for \(x\geq0\), \(u\) has the following form,
    \[
    u(x)=u(0)+\int_0^x \psi(t) dt=u(0)+\int_0^{|x|} \psi(t) dt.
    \]
	where $u(0)=\lim\limits_{t \rightarrow+0} u(t)$ and $\psi$ is a positive function on $(0, \infty)$. In other words, let $u$ be an absolute continuous function on $(0, \infty)$. Indeed, for any increasing function $u$ on $(0, \infty)$ there exists a sequence of absolutely continuous functions $\left\{\varphi_n\right\}$ such that $\lim\limits_{n \rightarrow \infty} \varphi_n(t)=u(t)$, $0 \leq \varphi_n(t) \leq u(t)$ a.e. $t>0$ and $\varphi_n(0)=u(0)$. Furthermore the functions $\varphi_n(t)$ are increasing, and besides
	$$
	\varphi_n(t)=\varphi_n(0)+\int_0^t \varphi_n^{\prime}(\tau) d \tau.
	$$

    Since \(u\) is an even function on \(\mathbb{R}\), when \(x < 0\), we have
    \[
    u(x)=u(-x)=u(0)+\int_0^{-x} \psi(t) dt=u(0)+\int_0^{|x|} \psi(t) dt.
    \]
    We define
	$$
	\Psi(t):=
	\begin{cases}
	\psi(t), &t>0\\
	0,		 &t\leq0.
	\end{cases}
	$$
	 Therefore, for \(x\in\mathbb{R}\),
    \begin{equation}\label{eq3.27}
    u(x)=u(0)+\int_{|t|\leq |x|} \Psi(t) dt.
    \end{equation}

    We start to estimate \(\|h_{\Phi,\beta}(f)\|_{L^q_u}\). By $(\ref{eq3.27})$, we have
    \[
    \begin{aligned}
    \|h_{\Phi,\beta}(f)\|_{L^q_u(\R)}
    &\leq
    \left(\int_{\R} |h_{\Phi,\beta}(f)|^q u(0)
     dx\right)^\frac{1}{q}
    +
    \left(\int_{\R} |h_{\Phi,\beta}(f)|^q\left(\int_{|t|\leq|x|} \Psi(t) dt\right)
     dx\right)^\frac{1}{q}\\
     &=:E_1+E_2.
    \end{aligned}
    \]
    From \((\ref{eq3.9})\) and that \(u\), \(v\) are even functions on \(\mathbb{R}\), increasing on \(\mathbb{R}_+\) and $q(1-\beta)>1$, it follows that 
    \[
    A
    \geq
    \sup_{t>0} \frac{[u(t)]^\frac{1}{q}}{[v(t)]^\frac{1}{p}}
    \left( \int_{|x|\geq t} |x|^{q(\beta-1)} dx \right)^\frac{1}{q} 
    \left(\int_{|x|\leq t} dx\right)^\frac{1}{p'}
    =
    2^{1-\beta}[q(1-\beta)-1]^{-\frac{1}{q}}
    \sup_{t>0} \frac{[u(t)]^\frac{1}{q}}{[v(t)]^\frac{1}{p}},
    \]
    Therefore, for all \(t>0\), we have
    \begin{equation}\label{eq3.10}
        u(t)\leq 2^{(\beta-1)q} A^q [q(1-\beta)-1][v(t)]^\frac{q}{p}.
    \end{equation}
    Since \(u\) and \(v\) are even functions, \((\ref{eq3.10})\) also holds for \(t < 0\). 
    
    By Lemma \ref{t3.4} with $(\ref{eq3.25})$, and \((\ref{eq3.10})\), we have
    \[
    \begin{aligned}
        E_1
        &=[u(0)]^\frac{1}{q}\left(\int_{\R}|h_{\Phi,\beta}(f)(x)|^qdx\right)^\frac{1}{q}\\
        &\leq 2^\frac{1}{q'} K_{\Phi,\beta,q} [u(0)]^\frac{1}{q}\left(\int_{\R}|f(x)|^p dx\right)^\frac{1}{p}\leq 2^\frac{1}{q'} K_{\Phi,\beta,q} \left(\int_{\R}|f(x)|^p [u(x)]^\frac{p}{q} dx\right)^\frac{1}{p}\\
        &\leq 2^\frac{1}{q'} K_{\Phi,\beta,q} 2^{\beta-1} A [q(1-\beta)-1]^\frac{1}{q}\left(\int_{\R}|f(x)|^pv(x)dx\right)^\frac{1}{p}.
    \end{aligned}
    \]
    Let us estimate \(E_2\):
    \[
    \begin{aligned}
        E_2
        &=\left( \int_{\R} |h_{\Phi,\beta}(f)(x)|^q\left( 
        \int_{\R} \Psi(t) \chi_{\{|t|\leq|x|\}}(t) dt \right) dx \right)^\frac{1}{q}\\
        &=\left( \int_{\R} \Psi(t) \int_{|x|\geq|t|} |h_{\Phi,\beta}(f)(x)|^q dx dt \right)^\frac{1}{q}\\
        &\leq 2^\frac{1}{q'}\left( \int_{\R} \Psi(t) \int_{|x|\geq|t|} \left| \int_{|y|\leq|t|} \frac{\Phi(|y|^{-1}x)}{|y|^{1-\beta}} f(y)
        dy\right|^q dx dt \right)^\frac{1}{q}\\
        &~~~~~~~~+
        2^\frac{1}{q'}\left( \int_{\R} \Psi(t) \int_{|x|\geq|t|} \left| \int_{|y|\geq|t|} \frac{\Phi(|y|^{-1}x)}{|y|^{1-\beta}} f(y)
        dy\right|^q dx dt \right)^\frac{1}{q}\\
        &=:E_{21}+E_{22}.
    \end{aligned}
    \]
    
    We estimate \(E_{22}\). By Lemma \ref{t3.4} with $(\ref{eq3.25})$, generalized Minkowski's inequality with \(\frac{q}{p}\geq1\), $(\ref{eq3.27})$ and  \((\ref{eq3.10})\), we have
    \[
    \begin{aligned}
        E_{22}
        &\leq 2^\frac{1}{q'} \left( \int_{\R} \Psi(t) \int_{\R} \left| \int_{\R} \frac{\Phi(|y|^{-1}x)}{|y|^{1-\beta}} f(y) \chi_{\{ |y|\geq|t|\}}(y) dy \right|^q dx dt \right)^\frac{1}{q}\\
        &\leq 2^\frac{2}{q'} K_{\Phi,\beta,q} \left( \int_{\R} \left( \int_{\R} |f(x)|^p [\Psi(t)]^\frac{p}{q} \chi_{\{|x|\geq|t|\}}(x) dx \right)^\frac{q}{p} dt \right)^\frac{1}{q}\\
        &\leq 2^\frac{2}{q'} K_{\Phi,\beta,q} \left( \int_{\R} |f(x)|^p \left( \int_{|t|\leq|x|} \Psi(t) dt \right)^\frac{p}{q} dx\right)^\frac{1}{p}\\
        &\leq 2^\frac{2}{q'} K_{\Phi,\beta,q} \left( \int_{\R} |f(x)|^p [u(x)]^\frac{p}{q} \right)^\frac{1}{p}\\
        &\leq 2^\frac{2}{q'} K_{\Phi,\beta,q} A 2^{\beta-1} [q(1-\beta)-1]^\frac{1}{q} \left(\int_{\R} |f(x)|^p v(x)
         dx\right)^\frac{1}{p}.
    \end{aligned}
    \]

    Finally, we estimate \(E_{21}\). If \(|x|\geq|t|\) and \(|y|\leq |t|\), then \(|y|^{-1}x\leq -1\) or \(|y|^{-1}x\geq1\). Therefore, by $(\ref{eq2.3})$ with $|y^{-1}x|\geq 1$, i.e., \(\Phi(|y|^{-1}x)\leq C_2(|y|^{-1}x)^{\beta-1}\), we obtain
    \begin{align}
        E_{21}
        &\leq 2^\frac{1}{q'} \left( \int_{\R} \Psi(t) \int_{|x|\geq|t|} \left( \int_{|y|\leq|t|} \frac{\Phi(|y|^{-1}x)}{|y|^{1-\beta}} |f(y)| dy \right)^q dx dt \right)^\frac{1}{q}\nonumber\\
        &\leq 2^\frac{1}{q'} C_2 \left( \int_{\R} \Psi(t) \left(\int_{|x|\geq|t|} |x|^{(\beta-1)q} dx\right) \left( \int_{|y|\leq|t|}  |f(y)| dy \right)^q dx dt \right)^\frac{1}{q}\nonumber\\
        &=2^\frac{1}{q'} C_2 \left[ \frac{2}{(1-\beta)q-1} \right]^\frac{1}{q} \left( \int_{\R} \Psi(t) |t|^{(\beta-1)q+1} \left(\int_{|y|\leq|t|} |f(y)| dy\right)^q dt \right)^\frac{1}{q}.\label{eq3.28}
    \end{align}

    Notice that, for any $\alpha>0$, by $q(1-\beta)>1$ and $(\ref{eq3.27})$, we obtain
    \begin{align}
        \left[\frac{2}{(1-\beta)q-1}\right] \int_{|t|\geq \alpha} \Psi(t) |t|^{(\beta-1)q+1} dt
        &=\int_{\R}\int_{\R} \Psi(t) |x|^{(\beta-1)q} \chi_{\{ |t|\geq\alpha \}}(t) \chi_{\{|x|\geq|t|\}}(x) dtdx\nonumber\\
        &\leq \int_{|x|\geq \alpha} |x|^{(\beta-1)q} \int_{ |t|\leq |x|} \Psi(t) dt dx\nonumber\\
        &\leq \int_{|x|\geq\alpha} \frac{u(x)}{|x|^{(\beta-1)q}} dx\label{eq3.29}\nonumber.
    \end{align}
    From this and $(\ref{eq3.9})$, it follows that
    \[
    \left[ \frac{2}{(1-\beta)q-1} \right]^\frac{1}{q}
    \sup_{\alpha>0}
    \left( \int_{|t|\geq \alpha} \Psi(t) |t|^{(\beta-1)q+1}dt \right)^\frac{1}{q}
    \left( \int_{|t|\leq \alpha} v(t)^{1-p'} dt \right)^\frac{1}{p'}
    \leq A<\infty,
    \]
    which implies that $(\ref{eq3.9})$ holds true for $u(t)=\frac{2}{(1-\beta)q-1}\Psi(t)|t|$ and $v(t)$. Therefore, by $(\ref{eq3.28})$ and Lemma \ref{t3.5}, we obtain
    \[
    E_{21}\leq A C_2  2^\frac{1}{q'} (p')^\frac{1}{p'} p^\frac{1}{q} \left( \int_{\R} |f(x)|^p v(x) dx \right)^\frac{1}{p}.
    \]
    The proof of Theorem \ref{t1.9} is completed.
\end{proof}

\begin{proof}[Proof of Theorem \ref{t1.10}]
    Necessity. Let \(\alpha>0\) be a fixed number. Suppose that \((\ref{eq2.5})\) holds. Let \(f\) be a function defined on \(\R\) and \({\rm supp}~(f)\subset (-\infty,-\alpha] \cup [\alpha,\infty)\). If \(|x|\leq \alpha\) and \(|y|\geq \alpha\), then \(|y|^{-1}|x|=||y|^{-1}x|\leq1\), thus, by \((\ref{eq2.4})\), we have
    \begin{align}
        \|h_{\Phi,\beta}(f)\|_{L^q_u(\R)}&=
        \left( \int_{\R}\left|\int_{\R} \frac{\Phi(|y|^{-1}x)}{|y|^{1-\beta}} f(y) dy\right|^q u(x) dx \right)^\frac{1}{q}\nonumber\\
        &\geq \left( \int_{|x|\leq \alpha} \left| \int_{|y|\geq \alpha} \frac{\Phi(|y|^{-1}x)}{|y|^{1-\beta}} f(y) dy \right|^q u(x) dx \right)^\frac{1}{q}\nonumber\\
        &\geq C_1' \left( \int_{|x|\leq \alpha} u(x) dx\right)^\frac{1}{q} \left(\int_{|y|\geq \alpha} \frac{f(y)}{|y|^{1-\beta}} dy \right). \label{eq3.21}
    \end{align}

    We choose the test function \(f=f_\alpha (x):=[|x|^{1-\beta}v(x)]^{1-p'}\chi_{\{|x|\geq \alpha\}}(x)\), thus
    \begin{equation}\label{eq3.22}
    \|f_\alpha \|_{L^p_v(\R)}=\left(\int_{|x|\geq \alpha} \frac{v(x)^{1-p'}}{|x|^{(1-\beta)p'}} dx\right)^\frac{1}{p}.
    \end{equation}
    Then, by $(\ref{eq3.21})$ with $f=f_\alpha$, $(\ref{eq2.5})$ and $(\ref{eq3.22})$ we obtain
    \[
	\begin{aligned}
    C_1' \left(\int_{|x|\leq \alpha} u(x) dx\right)^\frac{1}{q} \left( \int_{|x|\geq \alpha} 
    \frac{v(x)^{1-p'}}{|x|^{(1-\beta)p'}} dx \right) 
	&\leq \|h_{\Phi,\beta}(f)\|_{L^q_u(\R)} \\
	&\leq C\|f\|_{L^p_v(\R)}\\
	&= C\left( \int_{|x|\geq \alpha} \frac{v(x)^{1-p'}}{|x|^{(1-\beta)p'}} dx \right)^\frac{1}{p},
	\end{aligned}
    \]
    which implies that
    \[
    C_1' \left(\int_{|x|\leq \alpha} u(x) dx\right)^\frac{1}{q} \left( \int_{|x|\geq \alpha} 
    \frac{v(x)^{1-p'}}{|x|^{(1-\beta)p'}} dx \right)^\frac{1}{p'}\leq C .
    \]

    Sufficiency. By  $(\ref{eq2.4})$ and $q(1-\beta)>0$, we have
    \begin{align}
    K_{\Phi,\beta,q}
    &\leq \left( \int_{|t|\leq 1} \Phi(t)^\frac{1}{1-\beta} |t|^{-1+\frac{1}{q(1-\beta)}} dt \right)^{1-\beta} + \left( \int_{|t|\geq 1} \Phi(t)^\frac{1}{1-\beta} |t|^{\frac{1}{q(1-\beta)}-1} dt \right)^{1-\beta}\nonumber\\
    &\leq C_2' \left( \int_{|t|\leq 1}|t|^{\frac{1}{q(1-\beta)}-1} dt \right)^{1-\beta}
    +
     \left( \int_{|t|\geq1} \Phi(t)^\frac{1}{1-\beta} |t|^{\frac{1}{q(1-\beta)}-1} dt \right)^{1-\beta}<\infty.\label{eq3.26}
    \end{align}
    Therefore, the condition \((\ref{eq3.8})\) of Lemma \ref{t3.4} is satisfied.

    Same to \cite[Proof of Theorem 3.3]{2021Safa}, without loss of generality, we may assume that for \(x\geq0\), \(u\) has the following form:
    \[
    u(x)=u(\infty)+\int_x^\infty \psi(t) dt=u(\infty)+\int_{|x|}^\infty \psi(t) dt.
    \]
	where $u(\infty)=\lim _{t \rightarrow \infty} u(t)$ and $\psi$ is a positive function on $(0, \infty)$. Indeed, for any decreasing function $u$ on  $(0, \infty)$ there exists a sequence of absolutely continuous functions $\left\{\varphi_n\right\}$ such that $\lim _{n \rightarrow \infty} \varphi_n(t)=u(t), 0 \leq \varphi_n(t) \leq u(t)$ a.e. $t>0$ and $\varphi_n(\infty)=u(\infty)$. Furthermore the functions $\varphi_n(t)$ are decreasing, and besides

	$$
	\varphi_n(t)=\varphi_n(\infty)+\int_t^{\infty}\left(-\varphi_n^{\prime}(\tau)\right) d \tau.
	$$

    Since \(u\) is an even function on \(\mathbb{R}\), when \(x < 0\), we have
    \[
    u(x)=u(-x)=u(\infty)+\int_{-x}^{\infty} \psi(t) dt=u(\infty)+\int_{|x|}^{\infty} \psi(t) dt.
    \]
    We define 
	$$
	\Psi(t):=
	\begin{cases}
	\psi(t),&t<0\\
	0,              &t\geq0.
	\end{cases}
	$$
	Therefore, for \(x\in\mathbb{R}\),
    \begin{equation}\label{eq3.30}
    u(x)=u(\infty)+\int_{|t|\geq |x|} \Psi(t) dt.
    \end{equation}

    We start to estimate \(\|h_{\Phi,\beta}(f)\|_{L^q_u}\). By $(\ref{eq3.30})$, we have
    \[
    \begin{aligned}
    \|h_{\Phi,\beta}(f)\|_{L^q_u(\R)}
    &\leq
    \left(\int_{\R} |h_{\Phi,\beta}(f)|^q u(\infty)
     dx\right)^\frac{1}{q}
    +
    \left(\int_{\R} |h_{\Phi,\beta}(f)|^q\left(\int_{|t|\leq|x|} \Psi(t) dt\right)
     dx\right)^\frac{1}{q}\\
     &=:F_1+F_2.
    \end{aligned}
    \]
	From $(\ref{eq3.11})$, that $u,v$ are even function on $\R$, decreasing on $\R_+$ and $p'(1-\beta)>1$, it follows that
    \[
    B
    \geq
    2^{1-\beta} \left[ \frac{1}{(1-\beta)p'-1} \right]^\frac{1}{p'}
    \sup_{t>0} \frac{[u(t)]^\frac{1}{q}}{[v(t)]^\frac{1}{p}}.
    \]
    Therefore, for all \(t>0\), we have
    \begin{equation}\label{eq3.12}
        u(t)\leq B^q 2^{q(\beta-1)} [(1-\beta)p'-1]^\frac{q}{p'} [v(t)]^\frac{q}{p}.
    \end{equation}
    Since \(u\) and \(v\) are even functions, \((\ref{eq3.12})\) also holds for \(t < 0\). 
    
    By Lemma \ref{t3.4} with $(\ref{eq3.26})$, and \((\ref{eq3.12})\), we have
    \[
    \begin{aligned}
        F_1
        &=[u(\infty)]^\frac{1}{q}\left(\int_{\R}|h_{\Phi,\beta}(f)(x)|^qdx\right)^\frac{1}{q}\\
        &\leq 2^\frac{1}{q'} K_{\Phi,\beta,q} [u(\infty)]^\frac{1}{q}\left(\int_{\R}|f(x)|^p dx\right)^\frac{1}{p}\leq  2^\frac{1}{q'} K_{\Phi,\beta,q} \left(\int_{\R}|f(x)|^p [u(x)]^\frac{p}{q} dx\right)^\frac{1}{p}\\
        &\leq  2^\frac{1}{q'} K_{\Phi,\beta,q} B 2^{\beta-1} [(1-\beta)p'-1]^\frac{1}{p'} \left(\int_{\R}|f(x)|^pv(x)\right)^\frac{1}{p}.
    \end{aligned}
    \]
    Let us estimate \(F_2\):
    \[
    \begin{aligned}
        F_2
        &=\left( \int_{\R} |h_{\Phi,\beta}(f)(x)|^q\left( 
        \int_{\R} \Psi(t) \chi_{\{|t|\geq|x|\}}(t) dt \right) dx \right)^\frac{1}{q}\\
        &=\left( \int_{\R} \Psi(t) \int_{|x|\leq|t|} |h_{\Phi,\beta}(f)(x)|^q dx dt \right)^\frac{1}{q}\\
        &\leq 2^\frac{1}{q'}\left( \int_{\R} \Psi(t) \int_{|x|\leq|t|} \left| \int_{|y|\leq|t|} \frac{\Phi(|y|^{-1}x)}{|y|^{1-\beta}} f(y)
        dy\right|^q dx dt \right)^\frac{1}{q}\\
        &~~~~~~~~+
        2^\frac{1}{q'}\left( \int_{\R} \Psi(t) \int_{|x|\leq|t|} \left| \int_{|y|\geq|t|} \frac{\Phi(|y|^{-1}x)}{|y|^{1-\beta}} f(y)
        dy\right|^q dx dt \right)^\frac{1}{q}\\
        &=:F_{21}+F_{22}.
    \end{aligned}
    \]
    
    We estimate \(F_{21}\). By Lemma \ref{t3.4}  with $(\ref{eq3.26})$,  generalized Minkowski's inequality with  \(\frac{q}{p}\geq1\), $(\ref{eq3.30})$ and  \((\ref{eq3.12})\), we have
    \[
    \begin{aligned}
        F_{21}
        &\leq 2^\frac{1}{q'} \left( \int_{\R} \Psi(t) \int_{\R} \left| \int_{\R} \frac{\Phi(|y|^{-1}x)}{|y|^{1-\beta}} f(y) \chi_{\{ |y|\leq|t|\}}(y) dy \right|^q dx dt \right)^\frac{1}{q}\\
        &\leq 2^\frac{2}{q'} K_{\Phi,\beta,q} \left( \int_{\R} \left( \int_{\R} |f(x)|^p \chi_{\{|x|\leq|t|\}}(x) [\Psi(t)]^\frac{p}{q} dx \right)^\frac{q}{p} dt \right)^\frac{1}{q}\\
        &\leq 2^\frac{2}{q'} K_{\Phi,\beta,q} \left( \int_{\R} |f(x)|^p \left( \int_{|t|\geq|x|} \Psi(t) dt \right)^\frac{p}{q} dx\right)^\frac{1}{p}\\
        &\leq 2^\frac{2}{q'} K_{\Phi,\beta,q} \left( \int_{\R} |f(x)|^p [u(x)]^\frac{p}{q} \right)^\frac{1}{p}\\
        &\leq 2^\frac{2}{q'} K_{\Phi,\beta,q} B 2^{\beta-1} [(1-\beta)p'-1]^\frac{1}{p'} \left(\int_{\R} |f(x)|^p v(x)
         dx\right)^\frac{1}{p}.
    \end{aligned}
    \]

    Finally, we estimate \(F_{22}\). If \(|x|\leq|t|\) and \(|y|\geq |t|\), then \(|y|^{-1}|x|=||y|^{-1}x|\leq1\). Therefore, by $(\ref{eq2.4})$ with $|y^{-1}x|\leq 1$, i.e., \(\Phi(|y|^{-1}x)\leq C_2'\), we obtain
    \begin{align}
        F_{22}
        &\leq 2^\frac{1}{q'} \left( \int_{\R} \Psi(t) \int_{|x|\leq|t|} \left( \int_{|y|\geq|t|} \frac{\Phi(|y|^{-1}x)}{|y|^{1-\beta}} |f(y)| dy \right)^q dx dt \right)^\frac{1}{q}\nonumber\\
        &\leq 2^\frac{1}{q'} C_2' \left( \int_{\R} \Psi(t) \left(\int_{|x|\leq|t|}  dx\right) \left( \int_{|y|\geq|t|}  \frac{|f(y)|}{|y|^{1-\beta}} dy \right)^q dx dt \right)^\frac{1}{q}\nonumber\\
        &=2^\frac{1}{q'} C_2' \left( \int_{\R} \Psi(t) 2|t| \left(\int_{|y|\geq|t|} \frac{|f(y)|}{|y|^{1-\beta}} dy\right)^q dt \right)^\frac{1}{q}.\label{eq3.31}
    \end{align}
	Notice that, for any $\alpha>0$, by $(\ref{eq3.30})$, we obtain
    \begin{align}
         \int_{|t|\leq \alpha} \Psi(t) 2|t| dt
        &=\int_{\R}\int_{\R} \Psi(t) \chi_{\{|t|\leq \alpha\}}(t) \chi_{\{|x|\leq |t|\}}(x) dt dx\nonumber\\
        &\leq \int_{|x|\leq \alpha} \int_{|x|\leq |t|} \Psi(t) dt dx\nonumber\\
        &\leq \int_{|x|\leq\alpha} u(x) dx .\nonumber
    \end{align}
    From this and $(\ref{eq3.11})$, it follows that
    \[
    \sup_{\alpha>0}
    \left( \int_{|t|\leq \alpha} \Psi(t) 2|t|dt \right)^\frac{1}{q}
    \left( \int_{|t|\geq \alpha} \frac{v(t)^{1-p'}}{|t|^{(1-\beta)p'}} dt\right)^\frac{1}{p'}
    \leq B<\infty,
    \]
    which implies that $(\ref{eq3.11})$ holds true for $u(t)=2\Psi(t)|t|$ and $v(t)$. Therefore, by $(\ref{eq3.31})$ and Lemma \ref{t3.6}, we obtain
    \[
    F_{22}\leq B C_2' 2^\frac{1}{q'} (p')^\frac{1}{p'} p^\frac{1}{q} \left( \int_{\R} |f(x)|^p v(x) dx \right)^\frac{1}{p}.
    \]
    The proof of Theorem \ref{t1.10} is completed.
\end{proof}


\begin{lemma}\label{t3.1}
    Let \(0<p<\infty\) and \(-1<\alpha<\infty\). If \(f\in H^p_{|\cdot|^\alpha}(\R)\), then for any \(s\in (0,\infty)\),
    \[
    	\left\|s^{\frac{1+\alpha}{p}}f(s\cdot)\right\|_{H^p_{|\cdot|^\alpha}(\R)}
		=
		\|f\|_{H^p_{|\cdot|^\alpha}(\R)}.
    \]
\end{lemma}

\begin{proof}
    Lemma \ref{t3.1} is from \cite[Lemma 3.1]{2017Ruan} without proof. For the completeness, we give its proof. By the definition of \(H^p_{|\cdot|^\alpha}(\R)\) and substituting $st=\widetilde{t}$, $sx=\widetilde{x}$, we have
    $$
    \begin{aligned}
    \left\|s^{\frac{1+\alpha}{p}}f(s\cdot)\right\|_{H^p_{|\cdot|^\alpha}}
    &=\left[\int_{\R} \left( \sup_{0<\epsilon<\infty  }  \left|  \int_{\R} s^\frac{1+\alpha}{p}\varphi_\epsilon(x-t)f(st) dt\right|\right)^p |x|^\alpha dx\right]^\frac{1}{p}\\
    &=\left[ \int_{\R} \left( \sup_{0<\epsilon<\infty} \left| \int_{\R} s^{\frac{1+\alpha}{p}} \varphi_\epsilon(x-s^{-1}t)f(t)s^{-1} dt \right| \right)^p |x|^\alpha dx \right]^\frac{1}{p}\\
	&=\left[ \int_{\R} \left( \sup_{0<\epsilon<\infty} \left| \int_{\R} s^{\frac{1+\alpha}{p}} (s\epsilon)^{-1} \varphi\left(\frac{sx-t}{s\epsilon}\right)f(t) dt \right| \right)^p |x|^\alpha dx \right]^\frac{1}{p}
	\\
	&=\left[ \int_{\R} \left( \sup_{0<\epsilon<\infty} \left| \int_{\R} s^{\frac{1+\alpha}{p}} \varphi_{s\epsilon}(sx-t)f(t) dt \right| \right)^p |x|^\alpha dx \right]^\frac{1}{p}
	\\
	&=\left[ \int_{\R} \left( \sup_{0<\epsilon<\infty} \left| \int_{\R}  \varphi_{s\epsilon}(x-t)f(t) dt \right| \right)^p |x|^\alpha dx \right]^\frac{1}{p}
	\\
    &=\|f\|_{H^p_{|\cdot|^\alpha}(\R)}.
    \end{aligned}
    $$
\end{proof}

\begin{lemma}\label{t3.2}\cite[Lemma 3.2]{2017Ruan}
 Let $0<p\leq 1 \leq r<\infty$ and $w(x)=|x|^\alpha$ with $0 \leq \alpha<\infty$. Assume $m:=\frac{1+\alpha}{p}-1$ is an integer, $f \in C^m\left(\mathbb{R}^n\right)$ and $f^{(m)} \in L^{\infty}\left(\mathbb{R}\right)$. Then for any $1+\alpha<r<\infty$ if $\alpha>0$ or $1 \leq r<\infty$ if $\alpha=0$, we have $f \in \widetilde{\mathcal{L}}_w^{p, r}\left(\mathbb{R}\right)$ and
$$
\|f\|_{\widetilde{\mathcal{L}}_w^{p, r}(\R)} \lesssim \| f^{(m)} \|_{L^{\infty}(\R)} .
$$
\end{lemma}

\begin{lemma}\label{t3.3}\cite[Lemma 3.2]{2016Chen}
    Let \(0\leq \beta <1\). If \(\Phi\),µ \(\widehat{\Phi}\in L^1(\R)\), then for any \(f\in \mathscr{S}(\R)\),
    \[
    Hh_{\Phi,\beta}(f)=h_{H\Phi,\beta}(f),
    \]
where \(H\) is the Hilbert transform.
\end{lemma}

\begin{proof}[Proof of Theorem \ref{t1.2}]
    By Lemma \ref{t3.1}, we get
    \begin{align*}
        \|f\|_{H^p_{|\cdot|^\alpha(\R)}}&=\|s^{\frac{1+\alpha}{p}}f(s\cdot)\|_{H^p_{|\cdot|^\alpha}(\R)}\\
        &\gtrsim \left\| s^{\frac{1+\alpha}{p}} \int_{\R} \frac{\Phi(|y|^{-1}\cdot)}{|y|^{1-\beta}} f(sy) dy \right\|_{H^q_{|\cdot|^\gamma}(\R)}\\
        &=  \left\| s^{\frac{1+\alpha}{p}} \int_{\R} \frac{\Phi(|u|^{-1}s\cdot)}{|s^{-1}u|^{1-\beta}} f(u)  
        s^{-1} du\right\|_{H^q_{|\cdot|^\gamma}(\R)}\\
        &=s^{\frac{1+\alpha}{p}-\frac{1+\gamma}{q}-\beta}\left\|s^{\frac{1+\gamma}{q}}h_{\Phi,\beta}(f)(s\cdot)\right\|_{H^q_{|\cdot|^\gamma}(\R)}\\
        &=s^{\frac{1+\alpha}{p}-\frac{1+\gamma}{q}-\beta}\|h_{\Phi,\beta}(f)\|_{H^q_{|\cdot|^\gamma}(\R)}.
    \end{align*}
    Since the inequality holds uniformly on \(0<s<\infty\), we have
    \[
    \frac{1+\alpha}{p}-\frac{1+\gamma}{q}=\beta.
    \]
\end{proof}


\begin{proof}[Proof of Theorem \ref{t1.6}]
Suppose the following condition holds for the moment:
\begin{equation}\label{eq3.17}
\sup _{0<s<\infty}\left\|\varphi_s*\Phi_{|y|,\beta}(x)\right\|_{\mathcal{L}_{|\cdot|^\alpha}^{p,r}(\R)} \lesssim |x|^{\beta-\frac{1+\alpha}{p}}.
\end{equation}
where $\Phi_{|y|,\beta}(x):=\frac{\Phi(|y|^{-1}x)}{|y|^{1-\beta}}$. By Proposition \ref{rem2.1} (i) and $(\ref{eq3.17})$, we obtain
    \begin{align}
        |(\varphi_s*h_{\Phi,\beta}(f))(x)|
        &=\left| \int_{\R}\varphi_s*\Phi_{|y|,\beta}(x) f(y) dy \right|\nonumber\\
        &\lesssim \left\| \int_{\R}\varphi_s*\Phi_{|y|,\beta}(x) dt \right\|_{\mathcal{L}^{p,r}_{|\cdot|^\alpha}(\R)} \|f\|_{H^p_{|\cdot|^\alpha}(\R)}\nonumber\\
		&\lesssim \|f\|_{H^p_{|\cdot|^\alpha}(\R)} |x|^{\beta-\frac{1+\alpha}{p}}.\label{eq3.23}
    \end{align}
    Denote \(w(x):=|x|^\gamma\). For any \(\lambda>0\) , by $(\ref{eq3.23})$ and $\frac{1+\alpha}{p}-\frac{1+\gamma}{q}=\beta$, we obtain
    \begin{align}
        w\left( \left\{ x\in\R:\sup_{0<s<\infty} |(\varphi_s*h_{\Phi,\beta}(f)(x)|>\lambda \right\} \right)
        &\lesssim w\left(\left\{x\in\R: \|f\|_{H^p_{|\cdot|^\alpha}(\R)}|x|^{\beta-\frac{1+\alpha}{p}}>\lambda\right\}\right)\nonumber\\
        &=w\left(\left\{ x\in\R: |x|< \left(\frac{\|f\|_{H^p_{|\cdot|^\alpha}}}{\lambda}\right)^{\frac{q}{1+\gamma}} \right\}\right)\nonumber\\
        &\lesssim \left(\frac{\|f\|_{H^p_{|\cdot|^\alpha}(\R)}}{\lambda}\right)^q,\label{eq3.43}
    \end{align}
    which implies that
	\begin{equation}\label{eq3.15}
	\|h_{\Phi,\beta}(f)\|_{H^{q,\infty}_{|\cdot|^\gamma}(\R)}\lesssim \|f\|_{H_{|\cdot|^\alpha}^p(\R)}.
	\end{equation}
     Thus, to end the proof, it leaves to verify $(\ref{eq3.17})$.

	 From Proposition \ref{rem2.1} (ii) , Lemma \ref{t3.2} and $m=\frac{1+\alpha}{p}-1$, it suffices to verify 
    \begin{equation}\label{eq3.41}
    \sup_{0<s<\infty}\sup_{|y|>0}\left||x|^{-\beta}\frac{d^m}{dy^m} \int_{\R}\varphi_s(x-t)\frac{\Phi(|y|^{-1}t)}{|y|^{1-\beta}}dt\right|\lesssim |x|^{-m-1}.
    \end{equation}
    Without loss of generality, we assume that \(y > 0\). By Proposition \ref{pro2.5} with $|x|^\alpha\in A_\infty$, we choose $\Psi \in \mathscr{S}\left(\mathbb{R}\right)$ such that its support is in the unit ball $B(0,1)$ with $\Psi(0) \neq 0$ and let $\varphi$ be the inverse Fourier transform of $\Psi$, which implies that $\int_{\R} \varphi(x) dx\neq0$. By properties of Fourier Transform and $\widehat{\Psi}=\varphi$, we obtain
    \begin{equation}\label{eq3.40}
    \int_{\R} \varphi_s(x-t) \frac{\Phi(y^{-1}t)}{y^{1-\beta}} dt
	=y^\beta\varphi_s*\Phi_y(x)=y^\beta\int_{\R} \widehat{\varphi_s}(\xi) \widehat{\Phi_y}(\xi) e^{2\pi ix\xi} d\xi =\int_{\R} \Psi(s\xi) \widehat{\Phi}(y\xi) e^{2\pi ix\xi} d\xi. 
    \end{equation}
    Thus, by $(\ref{eq3.41})$ and $(\ref{eq3.40})$, it suffices to verify
    \[
    \sup_{0<s<\infty}\sup_{y>0}\left||x|^{-\beta}\frac{d^m}{dy^m} y^{\beta}\int_{\R}\widehat{\Phi}(y\xi)\Psi(s\xi) e^{2\pi ix\xi} d\xi\right|\lesssim |x|^{-m-1}.
    \]

    For any \(y>0\) and \(0<s<\infty\), if \(|x|\leq y\), by Leibniz's rule, $\Psi\in\mathscr{S}(\R)$, substituting $y\xi=\tilde{\xi}$, $-1<\beta-1<0$, $|x|\leq y$ and $(\ref{eq3.14})$, we have
    \begin{align}\label{eq3.5}
        &\left||x|^{-\beta} \frac{d^m y^\beta}{dy^m}  \int_{\R} \widehat{\Phi}(y\xi) \Psi(s\xi) e^{2\pi ix \xi} d\xi \right|\nonumber\\
        =&|x|^{-\beta}\left| \sum_{n=0}^m \binom{m}{n}\frac{d^{m-n}y^\beta}{dy^{m-n}} \int_{\R} \widehat{\Phi}^{(n)}(y\xi)\xi^n \Psi(s\xi) e^{2\pi ix\xi} d\xi \right|\\
        \lesssim& \left(\frac{y}{|x|}\right)^{\beta-1} |x|^{-1} \sum_{n=0}^m y^{1-m+n} \int_{\R} |\widehat{\Phi}^{(n)}(y\xi)\xi^n| d\xi\nonumber\\
        =&\left(\frac{y}{|x|}\right)^{\beta-1} |x|^{-1} \sum_{n=0}^m y^{-m} \int_{\R} |\widehat{\Phi}^{(n)}(\xi)\xi^n| d\xi\nonumber\\
        \lesssim& |x|^{-1-m}\nonumber .
    \end{align}

    If \(|x|>y\). By integration by parts \(m+1\) times on \(\xi\) in equation \((\ref{eq3.5})\), Leibniz's rule, $\supp \Psi\subset B(0,1)$, $\Psi\in\mathscr{S}(\R)$, $0<\beta<1$, $|x|>y$ and $(\ref{eq3.14})$, we obtain
    \[
    \begin{aligned}
        &\left| \frac{d^m}{dy^m} \frac{y^\beta}{|x|^\beta} \int_{\R} \widehat{\Phi}(y\xi) \Psi(s\xi) e^{2\pi ix\xi} d\xi \right|
		\\
		=&|x|^{-\beta}\frac{1}{|2\pi x|^{m+1}} \left|\sum_{n=0}^m \binom{m}{n}\frac{d^{m-n}y^\beta}{dy^{m-n}}\int_{\R} \frac{d^{m+1}}{d\xi^{m+1}} \left[\widehat{\Phi}^{(n)}(y\xi) \xi^n \Psi(s\xi)\right] e^{2\pi ix\xi} d\xi\right|
		\\
		\lesssim&  \frac{1}{|x|^{m+1+\beta}} \sum_{n=0}^m y^{\beta-m+n} \int_{\R} \left| \sum_{k=0}^n \binom{m+1}{k} \sum_{l=0}^{m+1-k}\binom{m+1-k}{l} \frac{d^{m+1-k-l}\Psi(s\xi)}{d\xi^{m+1-k-l}} \frac{d^l\widehat{\Phi}^{(n)}(y\xi)}{d\xi^l} \frac{d^k\xi^n}{d\xi^k} \right| d\xi
		\\
		\lesssim& \frac{1}{|x|^{m+1}} \sum_{n=0}^m y^{-m+n} \sum_{k=0}^n \sum_{l=0}^{m+1-k} \int_{\R} \left| \Psi^{(m+1-k-l)}(s\xi) s^{m+1-k-l} \widehat{\Phi}^{(n+l)}(y\xi) y^l \xi^{n-k} \right| d\xi
		\\
		=& \frac{1}{|x|^{m+1}} \sum_{n=0}^m y^{-m+n} \sum_{k=0}^n \sum_{l=0}^{m+1-k} \int_{|\xi|<s^{-1}} \left| \Psi^{(m+1-k-l)}(s\xi) \frac{(s\xi)^{m+1-k-l}}{ \xi^{m+1-k-l}} \widehat{\Phi}^{(n+l)}(y\xi) y^l \xi^{n-k} \right| d\xi
		\\ 
		\leq& \frac{1}{|x|^{m+1}} \sum_{n=0}^m y^{-m+n} \sum_{k=0}^n \sum_{l=0}^{m+1-k} \int_{\R} \left| \widehat{\Phi}^{(n+l)}(y\xi) y^l \xi^{-m-1+l+n} \right| d\xi
		\\
		=&\frac{1}{|x|^{m+1}} \sum_{n=0}^m \sum_{k=0}^n \sum_{l=0}^{m+1-k} \int_{\R} \left| \widehat{\Phi}^{(n+l)}(\xi) \xi^{n+l-m-1} \right| d\xi
		\\
		\lesssim& |x|^{-m-1}.
    \end{aligned}
    \]
\end{proof}

Now we introduce the Marcinkiewicz interpolation theorem on weighted Hardy spaces.

\begin{lemma}\label{t3.7}\cite[Theorem 1]{Krotov2023}
Let $0<p_0 \neq p_1 \leq \infty$ and $0<q_0 \neq q_1 \leq \infty$, $u,v\in A_\infty$ and let a quasilinear operator $T:H^{p_0}_u(\R)+H^{p_1}_u(\R)\to L^0_v(\R)$
satisfy the following conditions: there exist positive constants $M_0$ and $M_1$ such that the inequalities
$$
\nu\{x\in\R:|T f|>\lambda\} \leq\left(\frac{M_i}{\lambda}\|f\|_{H^{p_i}_u (\mathbf{\R})}\right)^{q_i}, \quad f \in H^{p_i}_v (\R), \quad i=0,1
$$
hold for every $\lambda>0$.
Let $\theta \in(0,1)$ and let $p,q\in (0,\infty)$ satisfying
$$
\frac{1}{p}=\frac{1-\theta}{p_0}+\frac{\theta}{p_1}, \quad \frac{1}{q}=\frac{1-\theta}{q_0}+\frac{\theta}{q_1},
$$
and $p\leq q$. Then there exists a constant $C=C(p_0,p_1,q_0,q_1,\theta)$ such that the inequality
$$
\|T f\|_{L^q_v(\R)} \leq C M_0^{1-\theta} M_1^\theta\|f\|_{H^p_u(\R)}
$$
holds for all functions $u \in H^p_u(\R)$.
\end{lemma}

\begin{proof}[Proof of Theorem \ref{t1.4}]
Since \(m=[\frac{1+\alpha}{p}-1]+1>\frac{1+\alpha}{p}-1\), there is a \(p_1\in (0,p)\) satisfying
\[
m=\frac{1+\alpha}{p_1}-1.
\]
Theorem \ref{t1.6} implies
\begin{equation}\label{eq3.1}
\|h_{\Phi,\beta}(f)\|_{H^{q_1,\infty}_{|\cdot|^\gamma}(\R)}\lesssim \|f\|_{H^{p_1}_{|\cdot|^\alpha}(\R)},
\end{equation}
where
\[
\frac{1+\alpha}{p_1}-\frac{1+\gamma}{q_1}=\beta.
\]
On the other hand, since \(\alpha>0\), we have \(p_0>1+\alpha > 1\). Also, by $p_0>\max\{\frac{1+\alpha}{1+\beta+\gamma},1+\alpha\}$, $\gamma>0$ and $\frac{1+\alpha}{p_0}-\frac{1+\gamma}{q_0}=\beta$, we have \(q_0>1\) and \(q_0>\frac{1+\gamma}{1-\beta}\). Therefore, \(L^{p_0}_{|\cdot|^\alpha}(\R)=H^{p_0}_{|\cdot|^\alpha}(\R)\) and \(L^{q_0}_{|\cdot|^\gamma}(\R)=H^{q_0}_{|\cdot|^\gamma}(\R)\). By Lemma \ref{t3.4}, we obtain
\begin{equation}\label{eq3.2}
\|h_{\Phi,\beta}(f)\|_{H^{q_0}_{|\cdot|^\gamma}(\R)}\lesssim \|f\|_{H_{|\cdot|^\alpha}^{p_0}(\R)}.
\end{equation}
From \(p_1<p\), $\frac{1+\alpha}{p_1}-\frac{1+\gamma}{q_1}=\beta$ and $\frac{1+\alpha}{p}-\frac{1+\gamma}{q}=\beta$, we have 
\[
p_1=\frac{1+\alpha}{\frac{1+\gamma}{q_1}+\beta}<\frac{1+\alpha}{\frac{1+\gamma}{q}+\beta}=p,
\]
which implies 
\begin{equation}\label{eq3.3}
    q_1<q.
\end{equation}
By \(q_0>\frac{1+\gamma}{1-\beta}\) , $\frac{1+\alpha}{p}-\frac{1+\gamma}{q}=\beta$, $0<\beta<1$, \(\alpha>0\) and \(0<p\leq1\), we obtain
\begin{equation}\label{eq3.4}
q_0>\frac{1}{1-\beta}\left(\frac{1+\alpha}{p}-\beta\right)q=\frac{1+\alpha}{1-\beta}\frac{q}{p}-\frac{\beta}{1-\beta}q>q\frac{\alpha}{1-\beta}>q.
\end{equation}
Then, for any \(f\in H^p_{|\cdot|^\alpha}(\R)\), by \((\ref{eq3.1})\), \((\ref{eq3.2})\), \((\ref{eq3.3})\), \((\ref{eq3.4})\), \(0<p_1<p\leq1<p_0\), Lemma \ref{t3.7} with quasilinear operator \(T := M^+_\phi \circ h_{\Phi,\beta}\), $p\leq q$ and \(\theta=(1/p_0 - 1/p)^{-1}(1/p_0 - 1/p_1)\), we obtain 
\[
\|h_{\Phi,\beta}(f)\|_{H^q_{|\cdot|^\gamma}(\R)}\lesssim \|f\|_{H^p_{|\cdot|^\alpha}(\R)}.
\]
\end{proof}

\begin{proof}[Proof of Theorem \ref{t1.7}]
Without loss of generality, assume that \(y>0\). Pick the function $g\in C^{2m+1}: [0,\infty)\to\R$ as in Theorem \ref{t1.7}. A known result is as follows \cite[(3.28)]{2017Ruan}:
    \begin{equation}\label{eq3.33}
    \frac{d^ng(y^2\xi^2)}{d y^n}=
    \begin{cases}
        \sum\limits_{j=\frac{n+1}{2}}^n g^{(j)}(\xi^2 y^2)\xi^{2j}Q_{2j-n}(y),&n \text{ is an odd integer},\\
        \sum\limits_{j=\frac{n}{2}}^n g^{(j)}(\xi^2 y^2)\xi^{2j}Q_{2j-n}(y),&n \text{ is an even integer},
    \end{cases}
    \end{equation}
    where \(Q_{2j-n}(y)\) is a polynomial satisfies
    \begin{equation}\label{eq3.34}
    |Q_{2j-n}(y)|\lesssim |y|^{2j-n} .
    \end{equation}

	The outline of the proof is the same with that of Theorem \ref{t1.6}, it suffices to prove
	$$
	\left||x|^{-\beta} \frac{d^m}{d y^m}\left(y^\beta\int_{\R} g(y^2\xi^2)\Psi(s\xi)e^{2\pi ix\xi} d\xi\right)\right|\lesssim |x|^{-m-1}.
	$$
 
    For \(0<s<\infty\) and \(y>0\), if \(|x|\leq y\), by Leibniz's rule, $(\ref{eq3.33})$, $(\ref{eq3.34})$, $\Psi=\widehat{\varphi}\in\mathscr{S}(\R)$, substituting $\xi^2y^2=\tilde{\xi}$, $-1<\beta-1<0$, $|x|\leq y$ and $(\ref{eq3.42})$, we obtain
    \[
    \begin{aligned}
        &\left||x|^{-\beta} \frac{d^m}{d y^m}\left(y^\beta\int_{\R} g(y^2\xi^2)\Psi(s\xi)e^{2\pi ix\xi} d\xi\right)\right|
		\\
        =&\left||x|^{-\beta}\sum_{n=0}^m\binom{m}{n}\frac{d^{m-n}y^\beta}{d y^{m-n}}\int_{\R} \sum_{j=\left[\frac{n+1}{2}\right]}^n g^{(j)}(\xi^2y^2)\xi^{2j}Q_{2j-n}(y) \Psi(s\xi) e^{2\pi i x \xi}  d\xi\right|
		\\
        \lesssim&|x|^{-\beta}\sum_{n=0}^m y^{\beta-m} \sum_{j=\left[\frac{n+1}{2}\right]}^n \int_{\R} |g^{(j)}(\xi^2y^2)\xi^{2j}y^{2j}| d\xi
		\\
		=& |x|^{-\beta}\sum_{n=0}^m y^{\beta-m} \sum_{j=\left[\frac{n+1}{2}\right]}^n \left( \int_0^\infty |g^{(j)}(\xi^2y^2)\xi^{2j}y^{2j}| d\xi+\int_{-\infty}^0 |g^{(j)}(\xi^2y^2)\xi^{2j}y^{2j}| d\xi\right) 
		\\
        =&\left(\frac{y}{|x|}\right)^{\beta-1}|x|^{-1}y^{-m}\sum_{n=0}^m \sum_{j=\left[\frac{n+1}{2}\right]}^n \int_{0}^\infty |g^{(j)}(\xi)\xi^{j-\frac{1}{2}}|d\xi
		\\
        \lesssim& |x|^{-m-1}.
    \end{aligned}
    \]

    For \(0<s<\infty\) and \(y>0\), if \(|x|>y\), by Leibniz's rule, $(\ref{eq3.33})$ and integration by parts, we obtain
    \begin{align}
        &\left||x|^{-\beta} \frac{d^m}{d y^m}\left(y^\beta\int_{\R} g(y^2\xi^2)\Psi(s\xi)e^{2\pi ix\xi} d\xi\right)\right|\nonumber
		\\
		=& |x|^{-\beta} \left| \sum_{n=0}^m\binom{m}{n} \frac{d^{m-n}y^\beta}{dy^{m-n}} \int_{\R} \frac{d^n g(y^2\xi^2)}{d y^n} \Psi(s\xi) e^{2\pi ix\xi} d\xi \right|\nonumber
		\\
		=& |x|^{-\beta} \frac{1}{|2\pi x|^{m+1}} \left| \sum_{n=0}^m\binom{m}{n} \frac{d^{m-n} y^\beta}{d y^{m-n}} \int_{\R} \frac{d^{m+1}}{d\xi^{m+1}} \left[\sum_{j=\left[\frac{n+1}{2}\right]}^ng^{(j)}(\xi^2y^2)\xi^{2j}\Psi(s\xi)\right] Q_{2j-n}(y) e^{2\pi ix\xi} d\xi \right|\nonumber
		\\
		=& |x|^{-\beta} \frac{1}{|2\pi x|^{m+1}} \left| \sum_{n=0}^m\binom{m}{n} \frac{d^{m-n} y^\beta}{d y^{m-n}} \int_{\R} \sum_{j=\left[\frac{n+1}{2}\right]}^n\sum_{k=0}^{\min\{2j,m+1\}}\binom{m+1}{k}\sum_{l=0}^{m+1-k} \binom{m+1-k}{l}\right.\nonumber
		\\
		&~~~~~~~~\times\left. \frac{d^{m+1-k-l}\Psi(s\xi)}{d \xi^{m+1-k-l}} \frac{d^l g^{(j)}(\xi^2y^2)}{d\xi^l} \frac{d^k\xi^{(2j)}}{d\xi^k} Q_{2j-n}(y) e^{2\pi ix\xi} d\xi \right|\nonumber
		\\
		=& |x|^{-\beta} \frac{1}{|2\pi x|^{m+1}} \left| \sum_{n=0}^m \frac{d^{m-n} y^\beta}{d y^{m-n}} \int_{\R} \sum_{j=\left[\frac{n+1}{2}\right]}^n\sum_{k=0}^{\min\{2j,m+1\}}\binom{m+1}{k}\sum_{l=0}^{m+1-k} \binom{m+1-k}{l}\right.\nonumber
		\\
		&~~~~~~~~\times\left. \frac{d^{m+1-k-l}\Psi(s\xi)}{d \xi^{m+1-k-l}} \sum_{i=\left[\frac{l+1}{2}\right]}^l g^{(i+j)}(\xi^2 y^2) y^{2i} Q_{2i-l}(\xi) \frac{d^k\xi^{(2j)}}{d\xi^k} Q_{2j-n}(y) e^{2\pi ix\xi} d\xi \right|.\label{eq1}
		\end{align}
		By $(\ref{eq1})$, $(\ref{eq3.34})$, $\supp\Psi\subset B(0,1)$, $\Psi\in\mathscr{S}(\R)$, substituting $\xi^2y^2=\tilde{\xi}$, $0<\beta<1$, $|x|>y$ and $(\ref{eq3.42})$, we further obtain
		\begin{align*}
		 &\left||x|^{-\beta} \frac{d^m}{d y^m}\left(y^\beta\int_{\R} g(y^2\xi^2)\Psi(s\xi)e^{2\pi ix\xi} d\xi\right)\right|
		\\
		\lesssim& |x|^{-\beta} \frac{1}{|x|^{m+1}} \sum_{n=0}^m y^{\beta-m+n} \sum_{j=\left[\frac{n+1}{2}\right]}^n \sum_{k=0}^{\min\{2j,m+1\}} \sum_{l=0}^{m+1-k}\sum_{i=\left[\frac{l+1}{2}\right]}^l
		\\
		&~~~~~~~~\times \int_{\R} \left| \Psi^{(m+1-k-l)}(s\xi) s^{m+1-k-l} g^{(i+j)}(\xi^2 y^2) y^{2i} \xi^{2i-l} \xi^{2j-k} y^{2j-n} \right| d\xi
		\\
		\lesssim& \frac{1}{|x|^{m+1}} \sum_{n=0}^m y^{-m+n} \sum_{j=\left[\frac{n+1}{2}\right]}^n \sum_{k=0}^{\min\{2j,m+1\}} \sum_{l=0}^{m+1-k}\sum_{i=\left[\frac{l+1}{2}\right]}^l
		\\
		&~~~~~~~~\times \int_{|\xi|<s^{-1}} \left| \Psi^{(m+1-k-l)}(s\xi)^{m+1-k-l}  g^{(i+j)}(\xi^2 y^2) \xi^{-m-1+2i+2j} y^{2i+2j-n} \right| d\xi
		\\
		\lesssim& \frac{1}{|x|^{m+1}} \sum_{n=0}^m y^{-m+n} \sum_{j=\left[\frac{n+1}{2}\right]}^n \sum_{k=0}^{\min\{2j,m+1\}} \sum_{l=0}^{m+1-k}\sum_{i=\left[\frac{l+1}{2}\right]}^l 
		\\
		&~~~~~~~~\times \int_{\R} \left|g^{(i+j)}(\xi^2y^2)\xi^{-m-1+2i+2j}y^{2i+2j-n}\right|d\xi
		\\
		=&|x|^{-m-1} \sum_{n=0}^m \sum_{j=\left[\frac{n+1}{2}\right]}^n \sum_{k=0}^{\min\{2j,m+1\}} \sum_{l=0}^{m+1-k}\sum_{i=\left[\frac{l+1}{2}\right]}^l \int_0^\infty \left|g^{(i+j)}(\xi)\xi^{i+j-\frac{m}{2}-1}\right|d\xi
		\\
        \lesssim& |x|^{-m-1}.
    \end{align*}
\end{proof}


\begin{proof}[Proof of Theorem \ref{t1.1}]
%
%
Suppose the following condition holds for the moment:
\begin{equation}\label{eq3.16}
\left\||\cdot|^{\beta-1}(H\Phi)(|\cdot|^{-1}x)\right\|_{L^{p,r}_{|\cdot|^\alpha}(\R)}\lesssim |x|^{\beta-\frac{1+\alpha}{p}}.
\end{equation}
	For any $x\in\R$, by  Lemma \ref{t3.3}, Proposition \ref{rem2.1} (i) and $(\ref{eq3.16})$, we obtain
    \begin{align}
        |Hh_{\Phi,\beta}(f)(x)|
        &=|h_{H\Phi,\beta}(f)(x)|\nonumber\\
        &=\left| \int_{\R} \frac{H\Phi(|y|^{-1}x)}{|y|^{1-\beta}} f(y)dy\right|\nonumber\\
        &\lesssim \|f\|_{H^p_{|\cdot|^\alpha}(\R)} \left\||\cdot|^{\beta-1}(H\Phi)(|\cdot|^{-1}x)\right\|_{\mathcal{L}^{p,r}_{|\cdot|^\alpha}(\R)}\nonumber\\
        &\lesssim \|f\|_{H^p_{|\cdot|^\alpha}(\R)} |x|^{\beta-\frac{1+\alpha}{p}}.\label{eq3.24}
    \end{align}
    Denote \(w(x):=|x|^\gamma\). Same to the estimate of $(\ref{eq3.43})$ with $(\ref{eq3.24})$ and $\frac{1+\alpha}{p}-\frac{1+\gamma}{q}=\beta$, we obtain
    \[
        w\left( \left\{ x\in\R:|Hh_{\Phi,\beta}(f)(x)|>\lambda \right\} \right)
        \lesssim \left(\frac{\|f\|_{H^p_{|\cdot|^\alpha}(\R)}}{\lambda}\right)^q,
    \]
    which implies that
	\[
    \|Hh_{\Phi,\beta}(f)\|_{L_{|\cdot|^\gamma}^{q,\infty}(\R)}\lesssim \|f\|_{H^p_{|\cdot|^\alpha}(\R)}.
    \]
    Thus, to end the proof, it leaves to verify $(\ref{eq3.16})$. 

	From Proposition \ref{rem2.1} (ii), Lemma \ref{t3.2} with $m=\frac{1+\gamma}{p}-1$, it suffices to verify 
    \[
    \sup_{|y|>0} \left| \frac{d^m}{dy^m} \left(\frac{|y|^{\beta-1}}{|x|^\beta}(H\Phi)(|y|^{-1}x)\right) \right| \lesssim |x|^{-m-1}.
    \]
    Without loss of generality, we assume that \(y > 0\). By the inverse Fourier transform and $(\ref{eq2.6})$, we obtain
    \begin{align}
        \frac{y^{\beta-1}}{|x|^\beta}(H\Phi)(y^{-1}x)
		&= \int_{\R}\frac{y^{\beta-1}}{|x|^\beta} \widehat{H\Phi}(\xi)e^{2\pi ix\xi y^{-1}} d\xi\nonumber\\
        &=-i  \frac{1}{|x|^\beta}\int_{\R} {\rm sgn}(\xi) \widehat\Phi(\xi) y^{\beta-1} e^{2\pi i x\xi y^{-1}} d\xi.\label{eq3.35}
    \end{align}    
    A known result is as follows \cite[Page 12]{2016Chen}:
    \begin{equation}\label{eq3.37}
    \frac{d^n e^{2\pi ix\xi y^{-1}}}{dy^n}=e^{2\pi ix\xi y^{-1}}\sum_{n<k\leq 2n} (-1)^n y^{-k} P_{k-n}(2\pi ix\xi),
    \end{equation}
    where \(n\geq1\) is an integer and $P_{k - n}$ is a polynomial satisfying
    \begin{equation}\label{eq3.38}
    |P_{k-n}(2\pi ix\xi)|\lesssim |2\pi x\xi|^{k-n} .
    \end{equation}
    
    For any \(y>0\), if \(|x|\leq y\), by $(\ref{eq3.35})$, Leibniz's rule and $(\ref{eq3.37})$, we have
    \begin{align*}
    &\left| \frac{d^m}{dy^m} \left(\frac{y^{\beta-1}}{|x|^\beta}(H\Phi)(y^{-1}x)\right) \right|\\
    =&\left| -i \frac{d^m}{dy^m} \left( \frac{1}{|x|^\beta}\int_{\R} {\rm sgn}(\xi) \widehat\Phi(\xi) y^{\beta-1} e^{2\pi i x\xi y^{-1}} d\xi\right) \right|\\
    =&\left|\frac{1}{|x|^\beta}\int_{\R} {\rm sgn}(\xi) \widehat{\Phi}(\xi)\sum_{n=0}^m\binom{m}{n}\frac{d^{m-n}y^{\beta-1}}{dy^{m-n}}\frac{d^ne^{2\pi ix\xi y^{-1}}}{dy^n} d\xi\right|\\
    =&\left|\frac{1}{|x|^\beta}\int_{\R} {\rm sgn}(\xi) \widehat{\Phi}(\xi)e^{2\pi ix\xi y^{-1}}\left[\frac{d^my^{\beta-1}}{dy^m}+\sum_{n=1}^m\binom{m}{n}\frac{d^{m-n}y^{\beta-1}}{dy^{m-n}}\sum_{n<k\leq 2n} \frac{(-1)^n}{y^k} P_{k-n}(2\pi ix\xi) \right] d\xi\right|\\
    \leq&\left|\frac{1}{|x|^\beta}\int_{\R} {\rm sgn}(\xi)\widehat{\Phi}(\xi) e^{2\pi ix\xi y^{-1}}\frac{d^my^{\beta-1}}{dy^m} d\xi\right|\\
    &~~~~~~~~+\left|\frac{1}{|x|^\beta}\int_{\R} {\rm sgn}(\xi) \widehat{\Phi}(\xi)\sum_{n=1}^m\binom{m}{n}\frac{d^{m-n}y^{\beta-1}}{dy^{m-n}}e^{2\pi ix\xi y^{-1}}\sum_{n<k\leq 2n} \frac{(-1)^n}{y^k} P_{k-n}(2\pi ix\xi) d\xi\right|\\
    =:&I+J.
    \end{align*}

    For I, by $-1<\beta-1<0$, $|x|\leq y$, \(\widehat{\Phi}\in C^{2m+1}(\R)\) and that the support of \(\widehat{\Phi}\) is compact, we have
    \[
        I\lesssim \frac{y^{\beta-1-m}}{|x|^\beta} \int_{\R} |\widehat{\Phi}(\xi)| d\xi=\left(\frac{y}{|x|}\right)^{\beta-1}y^{-m}|x|^{-1}\int_{\R}|\widehat{\Phi}(\xi)|d\xi \lesssim |x|^{-m-1}.
    \]

    For J, by $(\ref{eq3.38})$, \(k-n\geq1\), $-1<\beta-1<0$, $|x|\leq y$,  \(\widehat{\Phi}\in C^{2m+1}(\R)\) and that the support of \(\widehat{\Phi}\) is compact, we obtain
    \[
    \begin{aligned}
    J\lesssim& \frac{1}{|x|^\beta}\sum_{n=1}^m y^{\beta-1-m+n} \sum_{n<k\leq 2n} y^{-k}\int_{\R} |\widehat{\Phi}(\xi)(x\xi)^{k-n}| d\xi\\
    =& y^{-m-1}\left(\frac{y}{|x|}\right)^{\beta-1}\sum_{n=1}^m\sum_{n<k\leq 2n}\left(\frac{|x|}{y}\right)^{k-n-1} \int_{\R}|\widehat{\Phi}(\xi)\xi^{k-n}|d\xi\\
    \lesssim& |x|^{-m-1}.
    \end{aligned}
    \]

    For any \(y>0\), if \(|x|>y\), by $(\ref{eq3.35})$, \(\text{sgn}(\xi)=\text{sgn}(y\xi)\) with \(y > 0\), substituting \(y^{-1}\xi=\tilde{\xi}\), we have
    \begin{equation}\label{eq3.36}
        \frac{y^{\beta-1}}{|x|^\beta}(H\Phi)(y^{-1}x)=-i\frac{1}{|x|^\beta}\int_{\R} {\rm sgn}(\xi) \widehat{\Phi}(y\xi)y^\beta e^{2\pi ix\xi}d\xi.
    \end{equation}
	By $(\ref{eq3.36})$ and Leibniz's rule, we obtain
    \[
        \begin{aligned}
        \left| \frac{d^m}{dy^m} \left(\frac{y^{\beta-1}}{|x|^\beta}(H\Phi)(y^{-1}x)\right) \right|
        &=\left|\frac{d^m}{dy^m}\left(\frac{1}{|x|^\beta}\int_{\R} {\rm sgn}(\xi) \widehat{\Phi}(y\xi)y^\beta e^{2\pi ix\xi}d\xi\right)\right|\\
        &=\frac{1}{|x|^\beta} \left|\int_{\R} {\rm sgn}(\xi) \sum_{n=0}^m \binom{m}{n} \frac{d^{m-n}y^{\beta}}{dy^{m-n}} \widehat{\Phi}^{(n)}(y\xi) \xi^n e^{2\pi ix\xi} d\xi \right|\\
        &\leq \frac{1}{|x|^\beta} \left[ \left|\int_{0}^{+\infty}\sum_{n=0}^m \binom{m}{n} \frac{d^{m-n}y^{\beta}}{dy^{m-n}} \widehat{\Phi}^{(n)}(y\xi) \xi^n e^{2\pi ix\xi} d\xi \right|\right.\\
        &~~~~~~~~+\left.\left|\int_{-\infty}^0  \sum_{n=0}^m \binom{m}{n} \frac{d^{m-n}y^{\beta}}{dy^{m-n}} \widehat{\Phi}^{(n)}(y\xi) \xi^n e^{2\pi ix\xi} d\xi \right|\right].\\
        \end{aligned}
    \]
    By integration by parts \(m+1\) times on \(\xi\), the support of \(\widehat{\Phi}\) is compact, Leibniz's rule, substituting $y\xi=\tilde{\xi}$, $0<\beta<1$, $|x|>y$ and \(\widehat{\Phi}\in C^{2m+1}(\R)\), we obtain 
        \begin{align*}
        &\left| \frac{d^m}{dy^m} \left(\frac{y^{\beta-1}}{|x|^\beta}(H\Phi)(y^{-1}x)\right)       \right|
		\\
        \leq& |x|^{-\beta} \frac{1}{|2\pi x|^{m+1}} \left[ \left|\int_0^\infty \sum_{n=0}^m \binom{m}{n} \frac{d^{m-n}y^\beta}{dy^{m-n}} \frac{d^{m+1}}{d\xi^{m+1}} \left( \widehat{\Phi}^{(n)}(y\xi)\xi^n \right) e^{2\pi ix\xi} d\xi\right| \right.
		\\
		&~~~~~~~~+\left.\left|\int_0^\infty \sum_{n=0}^m \binom{m}{n} \frac{d^{m-n}y^\beta}{dy^{m-n}} \frac{d^{m+1}}{d\xi^{m+1}} \left( \widehat{\Phi}^{(n)}(y\xi)\xi^n \right) e^{2\pi ix\xi} d\xi\right|\right]
		\\
		=&|x|^{-\beta} \frac{1}{|2\pi x|^{m+1}} \left[ \left|\int_0^\infty \sum_{n=0}^m \binom{m}{n} \frac{d^{m-n}y^\beta}{dy^{m-n}} \sum_{k=m+1-n}^{m+1}\binom{m+1}{k} \frac{d^k\widehat{\Phi}^{(n)}(y\xi)}{d\xi^k} \frac{d^{m+1-k}\xi^n}{d \xi^{m+1-k}} e^{2\pi ix\xi} d\xi\right| \right.
		\\
		&~~~~~~~~+\left.\left|\int_{-\infty}^0 \sum_{n=0}^m \binom{m}{n} \frac{d^{m-n}y^\beta}{dy^{m-n}} \sum_{k=m+1-n}^{m+1}\binom{m+1}{k} \frac{d^k\widehat{\Phi}^{(n)}(y\xi)}{d\xi^k} \frac{d^{m+1-k}\xi^n}{d \xi^{m+1-k}} e^{2\pi ix\xi} d\xi\right|\right]
		\\
		\lesssim& |x|^{-\beta}|x|^{-m-1}\sum_{n=0}^m y^{\beta-m+n}\sum_{k=m+1-n}^{m+1} \int_{\R} \left|\widehat{\Phi}^{(n+k)}(y\xi)y^k\xi^{n-m-1+k}\right| d\xi
		\\
		=& \left(\frac{y}{|x|}\right)^\beta |x|^{-m-1}\sum_{n=0}^m\sum_{k=m+1-n}^{m+1} \int_{\R} \left|\widehat{\Phi}^{(n+k)}(\xi)\xi^{n-m-1+k}\right| d\xi
		\\
        \lesssim& |x|^{-m-1}.
        \end{align*}
\end{proof}

\begin{lemma}\cite[Theorem 9]{Muchenhoupt1973}\label{t100}
Suppose $1<p<\infty$, $w$ is nonnegative function and $H$ is the Hilbert transform, then the following statements are equivalent:

(1) $w\in A_p$;

(2) There exists a constant $C$ independent of $f$, such that
$$
\|H(f)\|_{L^p_w(\R)}\leq C \|f\|_{L^p_w(\R)}.
$$
\end{lemma}

\begin{proof}[Proof of Theorem \ref{t1.8}]
Since \(m=[\frac{1+\alpha}{p}-1]+1>\frac{1+\alpha}{p}-1\), there is a \(p_1\in (0,p)\) satisfying
\[
m=\frac{1+\alpha}{p_1}-1.
\]
From this and using Theorem \ref{t1.6}, it follows that
\begin{equation}\nonumber
\|Hh_{\Phi,\beta}(f)\|_{L^{q_1,\infty}_{|\cdot|^\gamma}(\R)}\lesssim \|f\|_{H^{p_1}_{|\cdot|^\alpha}(\R)},
\end{equation}
where
\[
\frac{1+\alpha}{p_1}-\frac{1+\gamma}{q_1}=\beta.
\]
On the other hand, since \(\alpha>0\), we have \(p_0>1+\alpha > 1\). Also, by $p_0>\max\{\frac{1+\alpha}{1+\beta+\gamma},1+\alpha\}$, $\gamma>0$ and $\frac{1+\alpha}{p_0}-\frac{1+\gamma}{q_0}=\beta$, we have \(q_0>1\) and \(q_0>\frac{1+\gamma}{1-\beta}\). Therefore, \(L^{p_0}_{|\cdot|^\alpha}(\R)=H^{p_0}_{|\cdot|^\alpha}(\R)\). Thus, by Lemma \ref{t100} with $|x|^\gamma\in A_{q_0}$ (by $(\ref{eq2})$ with $q_0>1+\gamma$) and Lemma \ref{t3.4}, we obtain 
$$
\|Hh_{\Phi,\beta}(f)\|_{L^{q_0}_{|\cdot|^\gamma}(\R)}\lesssim \|h_{\Phi,\beta}(f)\|_{L^{q_0}_{|\cdot|^\gamma}(\R)}\lesssim \|f\|_{L^{p_0}_{|\cdot|^\alpha}(\R)}=\|f\|_{H^{p_0}_{|\cdot|^\alpha}(\R)}.
$$
Then using the same proof of Theorem \ref{t1.4} and Lemma \ref{t3.7} with quasilinear operator $T:=H\circ h_{\Phi,\beta}$, we further obtain
$$
\|Hh_{\Phi,\beta}(f)\|_{L^q_{|\cdot|^\gamma}(\R)} \lesssim \|f\|_{H^p_{|\cdot|^\gamma}(\R)}.
$$
From this and Proposition \ref{pro2.1}, it follows that
$$
\|h_{\Phi,\beta}(f)\|_{H^q_{|\cdot|^\gamma}(\R)} \lesssim \|f\|_{H^p_{|\cdot|^\gamma}(\R)}.
$$ 
\end{proof}

\noindent{\bf Funding } This project is supported by Natural Science Foundation of Xinjiang Province (No.2024D01C40) and NSFC (No.12261083).

\noindent{\bf Data Availability} No datasets were generated or analysed during the current study.

\section*{Declarations}

\noindent{\bf Ethical Approval} Not applicable.

\noindent{\bf Competing interests} The authors declare no competing interests.

\end{document}